\documentclass{amsart}

\usepackage{amssymb,amsmath}
\usepackage{url}
\usepackage{amscd}
\usepackage{texdraw}
\usepackage[all]{xy}
\usepackage{fancyhdr}
\usepackage{amsmath}
\usepackage{graphicx}
\usepackage{subcaption}

\usepackage{amscd,color}
\usepackage[shortlabels]{enumitem}

\theoremstyle{plain}
\newtheorem{thm}{Theorem}[section]
\newtheorem*{starthm}{Theorem}

\newtheorem{cor}[thm]{Corollary}
\newtheorem{prop}[thm]{Proposition}
\newtheorem{lemma}[thm]{Lemma}
\newtheorem{definition}{Definition}

\newtheorem{remark}{Remark}[section]

\input txdtools

\newcommand{\calc}{{\mathcal C}}

\newcommand{\calf}{{\mathcal F}}

\newcommand{\calm}{{\mathcal M}}

\newcommand{\cals}{{\mathcal S}}

\newcommand{\calv}{{\mathcal V}}

\newcommand{\calslao}{{\mathcal S_{\lambda}^0}}

\renewcommand{\AA}{{\mathbb A}}
\newcommand{\CC}{{\mathbb C}}
\newcommand{\DD}{{\mathbb D}}

\newcommand{\RR}{{\mathbb R}}

\newcommand{\ZZ}{{\mathbb Z}}
\newcommand{\chat}{{\widehat{\mathbb C}}}

\newcommand{\rmm}{{\rm M}}

\renewcommand{\hat}{\widehat}

\newcommand{\la}{\lambda}

\newcommand{\Log}{\mathrm{ Log}}

\begin{document}
\title[Two Asymptotic Values]{Slices of Parameter Space for Meromorphic Maps with Two Asymptotic Values}

\author{Tao Chen, Yunping Jiang, and Linda Keen}

\address{}
\email{}


\thanks{This material is based upon work supported by the National Science Foundation. It is partially supported by a collaboration grant from the Simons Foundation (grant number 523341) and PSC-CUNY awards.  
}


\maketitle

\begin{abstract} This paper is part of a program to understand the parameter spaces of dynamical systems generated by meromorphic functions with finitely many singular values.  We give a full description of the parameter space for a specific family based on the exponential function that has precisely two finite asymptotic values and one attracting fixed point.  It represents a step beyond the previous work in \cite{GK} on degree 2 rational functions with analogous constraints: two critical values and an attracting fixed point.  What is interesting and promising for pushing the general program even further, is that, despite the presence of the essential singularity,  our new functions exhibit a dynamic structure as similar as one could hope to the rational case, and  that the philosophy of the techniques used in the rational case could be adapted.

\end{abstract}

\section{Introduction}
A general principle in complex dynamics is that singular values control the dynamical behavior.  There is now a long history of isolating interesting families of functions whose singular values can be parameterized in a way that allows one to understand how the dynamics varies across the family.  In practice, one constrains the number of singular values and the behavior of one or more of them--for example, by demanding that the orbit of one tends to an attracting fixed point.

This paper is  a step along the way to a general theory for meromorphic functions with finitely many singular values.  We adapt a technique developed by Douady, Hubbard and their students to study spaces of cubic polynomials, and used in \cite{GK} for rational maps of degree 2, in which the parameter space is modeled on the dynamic space of a fixed map in the family.  We will be looking at a family of meromorphic functions that are close enough to rational maps of degree 2 that there should be (and is) a direct similarity between the behaviors.  To put this in context, it helps to review some of the history.

The study of the parameter space for families of complex dynamical systems began with the family of quadratic polynomials.
They have one critical value whose behavior determines the dynamics and it is this behavior that is captured by the Mandelbrot set and its complement.   The next step was to study families with two  free critical values --- cubic polynomials and rational maps of degree two.  Moving out of the realm of rational functions and into that of  dynamics of transcendental functions, we see more substantial differences between entire and meromorphic functions than between polynomials and rationals.   Rational maps define  finite coverings of the plane, but transcendental maps define  infinite coverings. Moreover, while the poles of rational maps are no different from regular points,  the poles of (transcendental) meromorphic functions add a new flavor to the dynamics.   It turns out that there  are more similarities between the parameter space of   rational maps of degree 2 and that of the tangent family $\la \tan z$ than between  quadratic polynomials and the exponential family.    See e.g. \cite{DFJ,DK,FG,KK,RG,Sch}.   Is this similarity just good fortune, or is it suggestive of a more  general pattern of relationships with rational maps?

 Thanks to invariance under M\"obius transformations, in order to study rational maps of degree $2$, we can restrict our attention to maps of the form $(z+b+1/z)/\rho$ where $b \in \CC$ and $\rho \in \CC^*$.  This family is often called $Rat_2$ in the literature. These
functions fix infinity where the derivative (multiplier) is $\rho \neq 0$ and have {\em two  free} critical values, $(b \pm 2)/\rho$, rather than one as in the  quadratic  polynomial case.
Constraining  $\rho$ to lie in the punctured unit disk, $\DD^*$, makes infinity an attracting fixed point for all
values of $b$.  In the paper \cite{GK}, a structure theorem is proved for this family that is as close as one could hope to the earlier examples:

\medskip
\begin{starthm}[Structure Theorem for $Rat_2$]
Fix $\rho \in \DD^*$, and consider the family $(z+b+1/z)/\rho$ where $b \in \CC$. The $b$ plane is divided into three components by a bifurcation locus: two copies of the Mandelbrot set that meet at the origin and are symmetric about it, and a ``shift locus''.   For $b$ in either copy of the Mandelbrot set, one or the other critical value is attracted to infinity and the other is not.  In the shift locus, both critical values are attracted to infinity.
\end{starthm}

This paper looks at   the family of meromorphic functions whose members ``look like degree 2 rationals'':  they have two finite   omitted asymptotic values   $\lambda$ and $\mu$ and an attracting fixed point (in this case, at the origin)   with multiplier $\rho$:
\begin{equation}~\label{family}
f_{\la, \rho}(z) = \frac{e^z - e^{-z}}{\frac{e^z}{\la} - \frac{e^{-z}}{\mu}} \mbox{  where  } \frac{1}{\la} - \frac{1}{\mu}=\frac{2}{\rho}, \,\rho \in \mathbb{D}^*.
\end{equation}
We use $\calf_2$ to denote this family. Our main result is a structure theorem for the slice of the parameter space defined by a fixed $\rho \in \DD^*$, and those $\la$'s for which that $\rho$ has a corresponding $\mu$, namely $\la$ not equal to $0$ or $\rho/2$.  It is a direct analogue of the $Rat_2$ theorem:

\medskip
\begin{starthm}[Main Structure Theorem]~\label{main1}
For each $\rho \in \DD^*$,  the parameter slice, $\la \in \CC \setminus \{0, \rho/2 \}$ divides into three distinct regions: two copies of connected and full sets $\calm_{\la}$ and $\calm_{\mu}$ in which only one  of the asymptotic values  $\mu$  or $\la$ is attracted to the origin and a ``shift locus'' $\cals$ in which both asymptotic values are attracted to the origin. The shift locus $\cals$ is  conformally equivalent to a  punctured annulus.   The puncture is at the origin.  The other puncture of the parameter plane, $\rho/2$, is on the boundary of the shift locus.
\end{starthm}

We are able to give a description of the sets $\calm_{\la}$ and $\calm_{\mu}$.  Like the Mandelbrot set in $Rat_2$, $\calm_{\la}$ and $\calm_{\mu}$ contain hyperbolic components in which one or the other of the asymptotic values tends to a non-zero attracting cycle.  Within each component the functions are quasiconformally conjugate.  Components like this were first studied in \cite{KK} where they occur in the parameter plane of the tangent family $\la \tan z$.  There, and in other  later work,  (see \cite{KK, FK, CK}), it was proved that each component is a universal cover of $\DD^*$;  based on the computer pictures, these components were called {\em shell components}. Thus, unlike the Mandelbrot set, the hyperbolic components do not contain a ``center'' where the  periodic cycle contains the critical value and has multiplier zero.  Instead, they contain a distinguished boundary point with the property that as the parameter approaches this point, the limit of the multiplier of the periodic cycle attracting the asymptotic value is zero. It is thus called a ``virtual center''.

 Like the characterization of centers of the components of the Mandelbrot set in terms of the sequence of inverse branches that keep the critical value fixed, a virtual center $\la^*$ can also be characterized by the property that  there is some $n$ such that $f_{\la^*}^{n-1}(\la^*) = \infty$ or $f_{\la^*}^{n-1}(\mu(\la^*,\rho)) = \infty$; the point is thus also called a ``virtual cycle parameter of order $n$''.   In this paper we give a complete combinatorial description of the virtual cycle parameters:

\medskip
\begin{starthm}[Combinatorial Structure Theorem]~\label{main2}
 The virtual cycle parameters $\lambda_{{\mathbf k}_{n}}$ of order $n$ can all be labelled by  sequences ${\mathbf k}_n= k_{n-1} \ldots k_1$,
where $k_i \in \ZZ$, in such a way that each of the parameters $\lambda_{{\mathbf k}_{n}}$ is an accumulation point in $\CC$
of a sequence of parameters $\lambda_{{\mathbf k}_{n+1}} $ of order $n+1$ and  related to ${{\mathbf k}_{n}}$; that is, ${\mathbf  k}_{n+1} =k_{n-1} \ldots k_1k_{0,j} $, $j\in \ZZ$.
This combinatorial description  of the virtual cycle parameters determines combinatorial descriptions of the sets $\calm_{\la}$ and $\calm_{\mu}$. 
\end{starthm}

 In~\cite{ CJK19}, we proved a ``transversality theorem'' for functions in the tangent family. Combining the techniques in the proof of this theorem with the results here,  we prove

 \medskip
\begin{starthm}[Common Boundary Theorem]~\label{main3}
Every virtual cycle parameter is both a boundary point of a shell component and a boundary point of the shift locus.
Furthermore,  the dynamics of the family $\{f_{\la}\}$ is transversal at each of these virtual cycle parameters (see Definition~\ref{trans} and Remark~\ref{tranrmk}).
And even further,  the set of all virtual cycle parameters is dense in the common boundary of the shift locus $\cals$ and the sets $\calm_{\la}\cup \calm_{\mu}$.
\end{starthm}

The paper is broken down into two parts.   Part 1 provides the background we need and some of the basic facts about the dynamical systems for functions in $\calf_2$.  Part 2 contains the main results of the paper.

We begin by quickly reviewing the basic definitions and facts we need about the dynamics of meromorphic functions and a theorem of Nevanlinna's, theorem~\ref {Nev}, that characterizes the functions we work with in terms of their Schwarzian derivatives.  We next take a detailed look at $\calf_2$.  In particular, in \ref{Julia set dichotomy} we show that  there is a dichotomy in the  dynamics in this family analogous to that for quadratic polynomials: either the Julia set is a Cantor set or there is a connected ``filled Julia set''  analogous to the filled Julia set of a quadratic polynomial.  
 
 Part 2 begins with the description of the sets $\calm_{\la}$ and $\calm_{\mu}$ from the Main Structure Theorem and gives the definitions of virtual cycle parameters and virtual centers.     The combinatorial description of the virtual cycle parameters,  the Combinatorial Structure Theorem,  is given in section~\ref{combinatorics}.  Pictures of the parameter plane follow and the rest of the paper contains the  the proof of the Common Boundary  Theorem, which leads to  a detailed discussion of the shift locus and the rest of Main Structure Theorem.
 
We would like to thank the reviewer for his or her careful reading  on the first version of this paper.  We have taken the comments into account in this version and it is a real improvement.

 \part{Background}
\section{Basic Dynamics}
\label{Basic Dynamics}
 Here we give the basic definitions, concepts and notations we will use.  When we say a function is meromorphic we mean that it is transcendental meromorphic.
We refer the reader to standard sources on meromorphic dynamics for details and proofs.  See e.g. \cite{Berg, BF,BKL1,BKL2,BKL3,BKL4,DK,KK}.

 We denote the complex plane by $\CC$, the Riemann sphere by $\hat\CC$ and the unit disk by $\DD$.  We denote the punctured plane by $\CC^* = \CC \setminus \{  0 \}$ and the punctured disk by $\DD^* = \DD \setminus \{  0 \}$.

\medskip
Given  a family of meromorphic functions, $\{ f_{\la}(z) \}$,  we look at the orbits of points formed by iterating the function $f(z)=f_{\la}(z)$.   If $f^k(z)=\infty$ for some $k>0$, $z$ is called a pre-pole of order $k$ --- a pole is a pre-pole of order $1$.  For meromorphic functions, the poles and pre-poles have  finite orbits that end at infinity.  The {\em Fatou set or Stable set, $F_f$}, consists of those points at which the iterates  $\{f_{\lambda}^{n}\}_{n=0}^{\infty}$  are well-defined and form a normal family in a neighborhood of each of them.   The Julia set $J_f$ is the complement of the Fatou set and  contains infinity as well as all the poles and pre-poles.  

\medskip
A point $z$ such that $f^n(z)=z$ is called {\em periodic}.  The minimal such $n>0$ is called the period. Periodic points are classified by their multipliers, $\nu(z)=(f^n)'(z)$ where $n$ is the period: they are repelling if $|\nu(z)|>1$, attracting if $0< |\nu(z)| < 1$,   super-attracting  if $\nu=0$ and neutral otherwise.  A neutral periodic point is {\em parabolic} if $\nu(z)=e^{2\pi i p/q}$ for some rational $p/q$.  The Julia set is the closure of the repelling periodic points.  For meromorphic $f$, it is also the closure of the pre-poles, (see e.g. \cite{BKL1}).

\medskip
If $D$ is a component of the Fatou set,  either $f^n(D) \subseteq f^m(D)$ for some integers $n,m$ or $f^n(D) \cap f^m(D) = \emptyset$ for all pairs of integers $m \neq n$.  In the first case $D$ is called {\em eventually periodic} and in the latter case it is called {\em wandering}.   The   periodic cycles of stable domains are classified as follows:
\begin{itemize}
\item Attracting or super attracting if the periodic cycle of domains contains an attracting or superattracting cycle in its interior.
\item Parabolic if there is a parabolic periodic cycle on its boundary.
\item Rotation if $f^n: D \rightarrow D$ is holomorphically conjugate to a rotation map.  Rotation domains are either simply connected or topological annuli.  These are called {\em Siegel disks and Herman rings} respectively.
\item Essentially parabolic, or Baker, if there is a point $z_{\infty} \in \partial D$ such that  $f^{n} (z_{\infty})$ is not well defined and for every $z \in D$,  $\lim_{k \to \infty} f^{nk}(z) = z_{\infty}.$
\end{itemize}

\medskip

A point $a$ is a {\em singular value} of $f$ if $f$ is not a regular covering map over $a$.
\begin{itemize}
\item    $a$ is a {\em critical value} if for some $z$, $f'(z)=0$ and $f(z)=a$.
\item    $a$ is an {\em asymptotic value}  if there is a path $\gamma(t)$, called an {\em asymptotic path},  such that $\lim_{t \to \infty} \gamma(t) = \infty$ and $\lim_{t \to \infty} f(\gamma(t))=a$.
\item The {\em set of singular values $S_f$} consists of the closure of the critical values and the asymptotic values.  The {\em post-singular set is
\[P_f= \overline{\cup_{a \in S_f} \cup_{n=0}^\infty f^n(a)  \cup \{\infty\}}. \]}
For notational simplicity, if a pre-pole $s$ of order $p$ is a singular value, $\cup_{n=0}^{p} f^n(s)$ is a finite set with $f^{p} (s)=\infty$.
\end{itemize}

\medskip
\begin{definition}\label{asymptract} If an asymptotic value $a$ is isolated, it has neighborhoods $U$ such that for at least one unbounded simply connected component $V$ of $f^{-1}(U\setminus \{a\})$, $f:V \rightarrow U\setminus \{a\}$ is a universal covering map and we call $V$ an {\em asymptotic tract for $a$}.  If $V_1$ and $V_2$ are asymptotic tracts for $a$, and $f: V_1 \cap V_2 \rightarrow U\setminus \{a\}$ is a universal covering map, we say $V_1$ and $V_2$ are equivalent.  The {\em multiplicity of the asymptotic value $a$} is the number of distinct equivalence classes of asymptotic tracts of $a$.  An asymptotic value is called simple if its multiplicity is $1$.
\end{definition}  

\medskip
Another important concept is
\begin{definition}
A map $f$ is {\em hyperbolic} if $J_f \cap {P_f} = \emptyset$.
\end{definition}

In rational dynamics, a map is hyperbolic if it satisfies an expansion property on its Julia set;  that is, (see \cite{Mil}),  there exist constants $c>0$ and $K>1$ such that for all $z$ in a neighborhood $V \supset J(f)$,   $|(f^{n})'(z)|> c K^n$.  For such maps, 
this property is equivalent to the condition that $J_f \cap {P_f} = \emptyset$.  

Because the Julia set of a meromorphic function is unbounded and its iterates have singularities at the prepoles we need a version of this condition tailored to transcendental maps.  
 We use   the following one proved in \cite{RS} which applies to hyperbolic functions in $\calf_2$.

 \begin{prop}[Rippon-Stallard] \label{RS}   If $S(f)$ is bounded and $\overline{PS(f)} \cap J(f) = \emptyset$,   then there exist two constants  $c>0$ and $K>1$ satisfying 
 \begin{equation}\label{Ripponstal}
 |(f^n)'(z) | > c K^n(|f^n(z)|+1)/(|z|+1).
 \end{equation}
 for all $z \in J(f) \setminus A_n(f)$ and all $n$ where $A_n(f)$ is the set of points where $f^n$ is not analytic (prepoles of lower order).  
 \end{prop}

\medskip
A standard result in dynamics is that each attracting, super-attracting, parabolic or Baker cycle of domains contains a singular value. Moreover, unless the cycle is superattracting, the orbit of the singular value is infinite and accumulates on the cycle, or  the orbit of $z_{\infty}$ associated with the Baker domain. The boundary of each rotation domain is contained in  the post singular set. (See e.g.~\cite{Mil}, chap 8-11 or~\cite{Berg}, Sect.4.3.)

\medskip
\subsection{Nevanlinna's theorem }\label{nev thm}

 Recall that the Schwarzian derivative is defined by
\[  S(f) = \big(\frac{f''}{f'} \big)' -\frac{1}{2} \big(\frac{f''}{f'} \big)^2.  \]
It satisfies the cocycle relation
\[  S(f \circ g) = S(f)( g')^2 + S(g).  \]
Since the Schwarzian derivative of a M\"obius transformations is zero,  solutions to the Schwarzian differential equation $S(f)(z)=P(z)$ are unique up to post-composition by a M\"obius transformation. See~\cite{Hil} and \cite{Nev1} for proofs.  

Nevanlinna's theorem characterizes  transcendental functions with finitely many singular values and finitely  many critical values
 in terms of their Schwarzian derivatives. 

\begin{thm}[Nevanlinna,\cite{Nev1}, Chap XI,  \cite{Hil}] \label{Nev}
Every meromorphic function $g$ with   $p < \infty$ asymptotic values and no critical values  has the property that its Schwarzian derivative is a polynomial function of degree $p -2$.
Conversely,  for every polynomial function $P(z)$ of degree $ p-2$,  the solution to
the Schwarzian differential equation $S(g)=P(z)$ is a meromorphic function with exactly $p$ asymptotic values and no critical points. The only essential singularity is at infinity.
\end{thm}

A summary  of the proof is given in \cite{DK1} where the behavior of the function in a neighborhood of infinity is described.  There are $p$ equally spaced asymptotic tracts separated by Julia directions along which the poles tend asymptotically to infinity.   An immediate corollary of the theorem is

\begin{cor}\label{Nevcor}  If $f$ is a meromorphic functions with $p$ finite simple asymptotic values and no critical values and $h$ is a homeomorphism of the complex plane $\CC$ such that $g=h^{-1} \circ f \circ h$ is holomorphic (meromorphic), then $S(g)$ is a polynomial of degree $p$. 
\end{cor}

In \cite{DK1}, this corollary is used to prove that if $f$ has polynomial Schwarzian derivative and all its asymptotic values are finite, then $f$ cannot have a Baker domain.
    
\section{Functions with two simple asymptotic values}
\label{2avs}

Our focus in this paper is on parameter spaces of meromorphic functions with two finite simple asymptotic values and no critical values.  By theorem~\ref{Nev}, such functions are characterized  by the property that they have a constant Schwarzian derivative.  Each asymptotic value is simple and there are exactly two distinct non-equivalent asymptotic tracts.  We  denote this family by ${\mathbb F}_{2}$.

It is easy to compute that $S(e^{2kz})= -2k^2$ and therefore that the most general solution to the equation $S(f)=-2k^2$ is
\begin{equation}\label{basicf} f(z)= \frac{a e^{kz} + be^{-kz}}{ce^{kz}+de^{-kz}},  \quad  ad - bc  \neq 0,
\end{equation}
and its asymptotic values are $\{ a/c, b/d \}$.  Note that both of them are omitted values. According to theorem~\ref{Nev}, the converse is true too. 
Moreover,  by corollary~\ref{Nevcor}, the solution to is unique up to post composition by an affine map.  Precomposition by an affine map multiplies the constant $k$ by the  scaling factor. 

Functions of the form~(\ref{basicf}) have a single essential singularity at infinity and their dynamics are invariant under affine conjugation. If one of the asymptotic values is equal to infinity, $c$ or $d=0$ and the family is the well-studied exponential family. See e.g. \cite{DFJ, RG}. The dynamics are quite different if both asymptotic values are finite and here we restrict ourselves to this situation.

Because we assume the asymptotic values are finite, neither $c$ nor $d$ can be zero. We  choose a representative of an equivalence class, where the equivalence relation is  defined by  affine conjugation, such that $k=1$ and $f(0)=0$; this implies $b=-a$. 
If the asymptotic values are $\la$ and $\mu$ we have
\[  f_{\la, \mu}(z) = \frac{e^{z}-e^{-z}}{\frac{e^{z}}{\la}-\frac{e^{-z}}{\mu}}   \]
where $\la, \mu \in \CC^*$.   If $f'(0)= \rho \in {\mathbb C}^{*}$ we have the relation
\begin{equation}\label{eqn - mult}
\frac{1}{\la} - \frac{1}{\mu} = \frac{2}{\rho}.
\end{equation}

 We still have the freedom to conjugate by the affine map $z \to -z$ so we see that the maps $f_{\la, \mu}(z)$ and $f_{\la', \mu'}(z)=-f_{\la, \mu}(-z)$ have the same dynamics.   That is, 
 $$f_{\la', \mu'}(z)=\frac{e^{z}-e^{-z}}{\frac{e^{z}}{\la'}-\frac{e^{-z}}{\mu'}}=\frac{e^z-e^{-z}}{\frac{e^z}{-\mu}-\frac{e^{-z}}{-\lambda}},$$
 where 
 \begin{equation}\label{eqn - mult}
\frac{1}{\la'} - \frac{1}{\mu'} = \frac{2}{\rho}=-\frac{1}{\mu} + \frac{1}{\la}.
\end{equation}
Set $f_{\la,\mu} \sim f_{\la',\mu'}$ if $\la'=-\mu, \mu'=-\la$ and use this equivalence relation to define the space of pairs of functions: 
$$
\hat{\mathbb{F}}_{2} =\Big\{ f_{\lambda, \rho} (z)=  \frac{e^{z}-e^{-z}}{\frac{e^{z}}{\la}-\frac{e^{-z}}{\mu}} \; \;\Big| \;\; \rho \in \CC^{*},\;\;\lambda \in \CC^*\setminus \{\rho/2\}, \;\; \frac{1}{\la} - \frac{1}{\mu} = \frac{2}{\rho} \Big \} \Big/ \sim.
$$
Note that each pair  of complex numbers $(\la, \rho)$ uniquely determines a pair of functions so that we also use $\hat{\mathbb F}_{2}$   to denote the moduli space of ${\mathbb F}_{2}$.

\medskip
Because of the ambiguity left by the normalization, it is difficult to study $\hat{\mathbb{F}}_2$ directly.  This situation is similar to the space $Rat_2$ of rational functions of degree $2$ with a fixed point at infinity discussed in the introduction.  The affine conjugation $z \to -z$  identifies maps with the same dynamics and sends the parameter $b$ to $-b$.  Thus  the $(b, \rho)$ space  is a 2-fold covering map of the space of functions.  To understand the role of the parameters, however, it is easier to work in this covering space.  This can be done by {\em marking} the singular points and choosing a ``preferred'' point.  In \cite{GK}, the preferred point was taken as $R(+1)$.  The conjugation $z \to -z$ interchanged the marking and corresponded to the involution $b \to -b$ in the lifted parameter space of functions with marked critical values.  See e.g. \cite{M1, GK} for more details.

We proceed in the same way here. To mark the asymptotic values, we choose $\la$ as the ``preferred" value and $\mu$ as the ``non-preferred" value.   That is, $\la = \lim_{t \to \infty}f_{\la,\rho}(\gamma(t))$ where $\Re \gamma(t) \to +\infty$.   We call the space with marked asymptotic values $\mathbb{F}_2$. 
Again the marked space is a 2-fold cover of the space of functions.  Note that if $\la=\infty, \mu=-\rho/2$ and  if $\mu=\infty, \la=\rho/2$.     
 
 Because the stable dynamics of functions in $\mathbb{F}_2$ are controlled by the behavior of the orbits of the asymptotic values, it will be convenient to choose a one dimensional ``slice" in this covering space of $\mathbb{F}_2$ in such a way that at each point in the slice, the orbit of one   asymptotic value has  fixed dynamics. One way to do this is to require that both asymptotic values have similar behavior.  For example, if $\mu=-\lambda$, so that $\la=\rho$, the slice obtained is the tangent family $ f_{\rho}(z)=\rho \tanh z = i \rho \tan (iz)$.  The properties of this slice have been investigated in \cite{KK,  CJK19}.  

 \subsection{The space $\calf_2$.}
 \label{calf2}
In this paper, we begin with the $2$ dimensional subfamily $\calf_2 \subset \mathbb F_2$ where $\rho$ is in the punctured unit disk $\DD^*$.  This means that the origin is an attracting fixed point so the orbit of at least one of the asymptotic values converges to zero.    It may be either the preferred asymptotic value $\la$ or not.   We can parameterize this subspace as 
\[ \calf_2 = \{ f_{\la,\rho} \} =\big( \CC \setminus \{0,\rho/2\} \big) \times \DD^* \]

Each $\rho \in \DD^*$ defines a one dimensional slice we denote by   $\calf_{2,\rho}$.   This is a ``dynamically natural slice'' in the sense of \cite{FK} because one asymptotic value is always attracted to the origin where the multiplier is fixed and the other is free.   We choose the asymptotic value $\la$ as   parameter for our slice; either it, or $\mu(\la)$ (determined by equation~(\ref{eqn - mult})) is the free asymptotic value. Note that because of equation~(\ref{eqn - mult}), when $\la = \rho/2$, $\mu=\infty$ and the function is in the exponential family, not our family.  Also, if $\la=0$, the function reduces to the constant $0$.    The  points of the slice are denoted  by $\la$, $f_{\la}$ or $f_{\la,\rho}$ if we want to emphasize the dependence on $\rho$;    if the context is clear, for readablity we use $f$.   We will prove these slices all have the same structure.

 Simple calculations show
 \[ f_{\la,\mu, \rho}(-z)=f_{\mu,\la,-\rho}(z) \mbox{  and  }  f_{\la,\mu, \rho}(-z)=f_{-\mu,-\la,\rho}(z) \]
  so that interchanging the asymptotic values $\la$ and $\mu$ changes  multiplier from $\rho$ to $-\rho$;   interchanging and negating the asymptotic values changes the marking.
  
\subsection{Fatou components for $f_{\la} \in \calf_{2}$}\label{fatoucomps}

For any $f_{\la} \in \calf_{2}$, the origin is an attracting fixed point with multiplier $\rho$.   Denote its attracting basin (which is non-empty) by $A_{\la}$.

\begin{prop}\label{compinv}
The attracting basin $A_{\la}$ is completely invariant.
\end{prop}
\begin{proof}
Since the origin is fixed, It is sufficient to prove  that its immediate basin of attraction $I_{\la} \subset A_{\la}$ is backward invariant.

On a neighborhood  $N \subset I_{\la}$ of the origin,  we can define a uniformizing map $\phi_{\la}(z)$ such that $\phi_{\la}(0)=0$,  $\phi_{\la}'(0)=1$ and $\phi_{\la} \circ f_{\la} = \rho \phi_{\la}$.      It extends by analytic continuation to the whole immediate attractive basin $I_{\la}$.
 Denote by  $O_{\la}$  the largest neighborhood of the origin  on which  $\phi_{\la}$ is injective.  One (or both) of the asymptotic values must be on the boundary of $O_{\la}$.   Assume for argument's sake that $\mu \in  \partial O_{\la}$.
 Choose a path $\gamma$ joining $0$ to $\mu$ in $O_{\la}$.  If $g$ is any inverse branch of $f_{\la}$,   then  $g(O_{\la})$ contains a path joining $g^{-1}(0)$ to infinity that passes through the asymptotic tract $\mathcal A_{\mu}$ of $\mu$.   Thus all these paths are contained in the same component of $f_{\la}^{-1}(O_{\la})$.  Therefore this component contains all the pre-images of $0$,  and since one branch fixes $0$, this component is $I_{\la}$.   It follows that
  $ I_{\la}$ is backward invariant and $I_{\la}=A_{\la}$.
   \end{proof}

\medskip
\begin{remark}  
The main point in the above argument is that whenever there is an attracting cycle, its basin contains a singular value, which, in this family is an omitted asymptotic value, and thus the component of the basin containing the asymptotic value can have only one preimage and it contains the asymptotic tract.   Above, because the attracting cycle consisted of the fixed point $0$, these components coincided.  If there is a second, non-zero attracting cycle, and  period of the cycle is one, it too has an invariant basin and the Fatou set consists of two completely invariant components.  If the period is greater than one, the component containing the asymptotic value and its preimage  are distinct and the full  basin contains infinitely many components; all but the component containing the asymptotic value have infinitely many preimages.  
 \end{remark}

\subsubsection{No Herman rings} 

Here we digress to prove a proposition about the non-existence of Herman rings for slightly larger classes of functions than $\calf_2$.  See also~\cite{Nay} for a similar study. 
%Although either would suffice to for $\calf_2$, each gives insight into how the asymptotic values affect the dynamics of the meromorphic functions so we include them both. 

\begin{prop}\label{NHR1}
 Suppose $f$ is a meromorphic function with an attracting fixed point whose basin of attraction has the properties that it contains at least one asymptotic value and  is completely invariant.  Then $f$ cannot have a Herman ring.  
 \end{prop}
\begin{proof}\footnote{ We thank the referee for suggesting this simplification of our original proof.}
We may assume without loss of generality that $f$ fixes the origin and attracts the asymptotic value $\mu$.  By hypothesis the basin of attraction  of the origin, $A_0$, is completely invariant and hence connected.  Moreover, because it contains an asymptotic value, it is unbounded. It follows that $\partial A_0 \subset J$ is also completely invariant.   A standard property of Julia sets is that if $z \in J$, then $J=\overline{\cup_{n \in \ZZ}f^n(z)},$ so that $\partial A_0=J$.   

If $f $ had a Herman ring $R$, its complement would consist of two components.  Moreover, because  the poles are dense in $J$, the bounded complementary component,  $B_R$, would contain a pole of some minimal order $m$ and   therefore   $B_R$ would contain an $m^{th}$ preimage of an asymptotic tract of $\mu$  so that it would intersect  $A_0$.   It follows that $A_0$ is disconnected which is a contradiction.

\end{proof}

\subsection{The parameter space trichotomy} 

Proposition~\ref{compinv} implies the following  trichotomy for $\calf_{2}$:
\begin{itemize}
\item $A_{\la}$ contains both asymptotic values: this is called the {\em shift locus} and denoted $\cals$.
\item  $A_{\la}$ contains only the preferred  asymptotic value $\la$:  in this case the other asymptotic value $\mu$ is not attracted to the origin and we call the set of such $\la$'s  $\calm_{\mu}$.   We denote the subset where $\mu$ is attracted to an attracting periodic cycle  by $\calm_{\mu}^0$.
\item $A_{\la}$ contains only the non-preferred asymptotic value $\mu$:  in this case the other asymptotic value $\la$ is not attracted to the origin and  we call the set of these $\la$'s  $\calm_{\la}$. We denote the subset where $\la$ is attracted to an attracting periodic cycle  by $\calm_{\la}^0$.
\end{itemize}

The maps in $\cals$, $\calm_{\la}^0$ and $\calm_{\mu}^0$ are hyperbolic because the orbits of their asymptotic values accumulate on attracting cycles.   The  connected components of these three subsets of parameter space are  thus called {\em hyperbolic components}. 

\medskip
As with the space $Rat_2$, there is an inversion of the space $\calf_{2}$ that 
 interchanges the regions $\calm_{\la}$ and $\calm_{\mu}$ and leaves $\cals$ invariant.

 Let ${\mathcal C}_0$ be the circle  in the $\la$ plane centered at  the parameter singularity $\rho/2$ with radius $|\rho/2|$ and let $D$ be the disk it bounds, punctured at the singularity $\rho/2$.  (See figure~\ref{circleofinversion}.) The inversion 
   \[ I(\la)=-\mu = \frac{\la}{2\la/\rho -1} \] leaves ${\mathcal C}_0$ invariant and interchanges $\la$ and $-\mu$.  
   \begin{prop}\label{calmu and calla}  If $f_{\la}^n(\la) \not\rightarrow 0$  as $n \to \infty$, then $f_{I(\la)}^n(I(\la)) \rightarrow 0$.  That is, the inversion interchanges the regions $\calm_{\la}$ and $\calm_{\mu}$ in the plane where only one of the asymptotic values goes to zero.
  \end{prop}
  
  \begin{proof}  Suppose   $f_{\la}^n(\la) \not\rightarrow 0$ so that $f_{\la}^n(\mu) \rightarrow 0$.  Since $ I(\la)=-\mu$ and $I(\mu)=-\la$, we can 
   write 
  \[ f_{-\mu}(-\mu)=\frac{e^{-2\mu}-1}{\frac{e^{-2\mu}}{-\mu} -\frac{1}{-\la}} = \frac{e^{2\mu}-1}{\frac{e^{2\mu}}{\la} -\frac{1}{\mu}}= f_{\la}(\mu) \]
and 
     \[ f_{-\mu}(-\la)=\frac{e^{-2\la}-1}{\frac{e^{-2\la}}{-\mu} -\frac{1}{-\la}} = \frac{e^{2\la}-1}{\frac{e^{2\la}}{\la} -\frac{1}{\mu}}= f_{\la}(\la). \]
  
    \end{proof}
  
 It follows that the inversion also preserves the region $\cals$ where both asymptotic values go to zero.  

  When $\rho$ is real, we can say more. 
   
 \begin{prop}\label{invpts}  Suppose $\rho$ is real  and $\la \in {\mathcal C}_0$. Then both $f_{\la}^n(\la) \rightarrow 0$ and $f_{\la}^n(\mu) \rightarrow 0$.    \end{prop}  

\begin{proof}   
Set $\la_1=I(\la)$. Then  $-\mu = I(\la) = \bar{\la}  =\la_1$.    Thus,
\[\overline{f_{\la} (\la)} = f_{\bar{\la}}(\bar{\la})=f_{\la_1}(\la_1); \] 
so if $f_{\la}^n(\la) \rightarrow 0$, $f_{\la_1}^n(\la_1) \rightarrow 0$.  

We can rewrite this as 
\[ f_{\la}(\mu)=f_{\la}(-\bar{\la}) = - \overline{f_{\la}(\la)}.\]
Therefore, either both asymptotic values iterate to zero or neither does.  Since  the origin is an attracting fixed point with multiplier $\rho$, at least one must and so they both do.   
\end{proof}

This proposition says when $\rho$ is real,  the region where both asymptotic values are attracted to zero contains the invariant circle of the inversion. 

Notice that the point $\la = \rho$ is on the circle ${\mathcal C}_0$. At that point we have $\mu = -\la$, $f_{\la} = \la \tanh z$ and  $I(\la)=\la$ so that it is a branch point of the double covering defined by the marking.    Moreover, because of the symmetry  both  asymptotic values are attracted to  zero.   

If $\la \in \calm_{\mu}^0 \mbox{ or } \calm_{\la}^0$, $f_{\la}$ has an attracting periodic cycle different from the origin.  This cycle has an attractive basin which we denote by $K_{\la}$ and $A_{\la} =\hat\CC \setminus \overline{K_{\la}}$.     Thus $\partial{K_{\la}}$ is the Julia set and $\overline{K_{\la}}$ is  the ``filled Julia set''.  Both of them are unbounded sets in $\CC$.  

    \subsection{The set $\overline{K_{\la}}$}\label{Julia set dichotomy}

  In \cite{KK}, it is proved that for $\rho \in \DD^*$, the Julia set $J_{\la}$ of the function $T_{\rho}(z)=\rho \tanh (z)$ is a  Cantor set.  Moreover, it is homeomorphic to a space consisting of finite and infinite sequences on an alphabet isomorphic to the natural numbers and infinity.  The finite sequences end with infinity. The homeomorphism conjugates $T_{\rho}$ to the shift map on this alphabet. See \cite{DK1, Mo} for details.    
 
At this point in this paper we can prove:

\begin{prop}\label{Cantordichot}
If $\Omega$ is the hyperbolic component of the $\la$ plane  containing  $f_{\la_0}=\rho \tanh z$
and  $\la \in \Omega$,  then the Julia set of $f_{\la}$ is a Cantor set.  If $\la \in \calm_{\la}$ or $\calm_{\mu}$  then  $\overline{K_{\la}}$ is  full.
\end{prop}

\begin{remark}  In theorem~\ref{thm:shift locus} we will prove that the shift locus $\cals$ is connected so that $\Omega=\cals$.  It will then follow that we have a dichotomy similar to that for quadratic polynomials.
\end{remark}

\begin{proof}   If $\la_0=\rho$, by symmetry, both $\lambda_0$ and $\mu_0=-\lambda_0$ are in $A_{\la}$.  By the results in \cite{KK}, the Julia set of $f_{\la_0}$ is a Cantor set.  Suppose $\la \in \Omega$, and let $\la(t)$, with $\la(0)=\la_0$ and $\la(1)=\la$, be a path in $\Omega$.
By standard arguments using quasiconformal surgery,  see e.g. \cite{McMSul,BF} and sections~\ref{map e} and~\ref{teichth}, we can construct quasiconformal homeomorphisms $g(t)$ conjugating $f_{\la_0}$ to $f_{\la(t)}$ that preserve the dynamics. Since the maps are hyperbolic, 
   the Julia sets of $f_{\la(t)}$ are  quasiconformally equivalent and thus also topologically equivalent.  

 Suppose now that $\la \in \calm_{\la}$ so that  $\la$ is not in $A_{\la}$.  The same argument works for $\la \in \calm_{\mu}$, interchanging the roles of $\la$ and $\mu$.  Take a generic small $r$, such that $\partial D_r(0)$ does not contain a point in the forward orbit of $\mu$. Then by definition,  $A_\lambda=\cup_{n\geq 1} f^{-n}(D_r(0))$ and $f^{-n}(D_r(0))\subset f^{-(n+1)}(D_r(0))$. Note that  because $\lambda\notin A_\lambda$, $f: f^{-(n+1)}(D_r(0))\to f^{-n}(D_r(0))\setminus\{\mu\}$ is a covering and so  $f^{-1}(D_r(0))$ is simply connected. Therefore $A_\lambda$ simply connected, which implies that it is complement $\overline{K_\lambda}$ is full.
\end{proof}
 
 Note that the argument above adapts easily to show that if $f_{\la}$ has a non-zero attracting or parabolic fixed point   the attracting basin of this fixed point is unbounded and completely invariant.  Other standard arguments, \cite{Mil}, show that if $f_{\la}$ has a neutral fixed point with a Siegel multiplier, its boundary must be contained in the post singular set.  Thus there are two completely invariant domains in the Fatou set separated by the Julia set.   An example of this is shown in figure~\ref{Juliaset1} where $\rho=2/3$ and $\la=2+2i$.  The yellow is the basin of $0$ and the blue is the basin of the fixed point  $2.25818 + 2.12632i$.   The proof of the Main Structure Theorem uses another example of a function with two attractive fixed points and its dynamic space is shown in figure~~\ref{modelbasin}.   
 
 \begin{figure}
     \centering
  \includegraphics[width=5in]{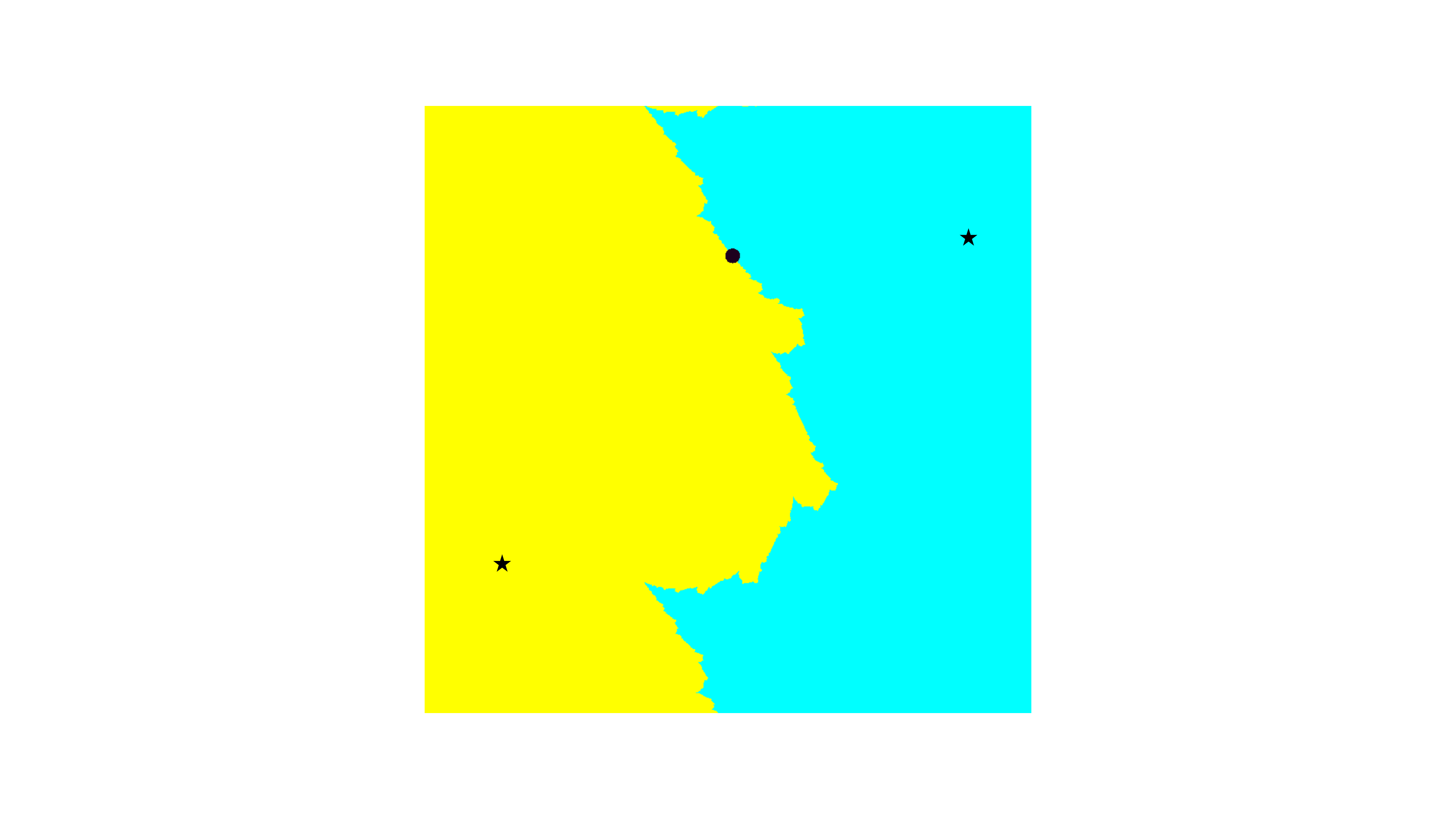}
  \caption{The dynamic plane of $f_{\la}$ with $\rho=2/3$ and $\la=2+2i$.  The fixed points are stars and the black dot is a pole. }
  \label{Juliaset1}
\end{figure}
 
\part{Properties of the Hyperbolic Components of the $\la$-plane.}

\section{Shell components:  Properties of $\calm_{\la}$ and $\calm_{\mu}$ }
\label{sec:shell components}
In this section, we work only with hyperbolic components in $\calm_{\la}^0$.  By propositions~\ref{calmu and calla} and~\ref{invpts}, the discussion for $\calm_{\mu}^0$ is essentially the same. 
 By definition, all the maps in $\calm_{\la}^0$ are hyperbolic; $\calm_{\la}^0$ consists of components in which standard arguments (see e.g. \cite{BF}) show any two functions corresponding to parameters in the component are quasiconformally conjugate.  Following \cite{FK} we call these  components {\em Shell Components}.    In that paper,   more general functions were considered and the properties of the shell components were described.  Here we summarize what we need from that description.
 We begin with some definitions.

\subsection{Virtual Cycle Parameters and Virtual Centers}\label{sec:vcs}

Let $\Omega$ be a hyperbolic component in $\calm_{\la}$ and let $\la \in \Omega$.  Both $\la$ and $\mu$ are attracted by attracting cycles of $f_{\la}$, and since $\la \in \calm_{\la}$, $\mu$ is attracted to the origin and $\la$ is attracted to a different cycle of order $n \geq 1$.  Since all the $f_{\la}$, $\la \in \Omega$ are quasiconformally conjugate,  all the functions in $\Omega$  have non-zero attracting cycles of  the same period, say $n$.    We say {\em $\Omega$ has period $n$} and where appropriate, denote it by $\Omega_n$.

We need the following definitions:

\begin{definition}  If $\la \in \calf_2$ and there exists an integer $n>1$ such that either $f_{\la}^{n-1}(\la)=\infty$ or $f_{\la}^{n-1}(\mu) = \infty$,  then $\la$ is called a {\em virtual cycle parameter}.  In the first case
  set $a_1=\la$ and in the second case set $a_1=\mu$.  Next set  $a_{i+1}=f_\la(a_i)$ where $i$ is taken modulo $n$ so that $a_{0}=\infty$.  We call the set  $\mathbf{a}=\{a_1, a_2, \ldots, a_{n-1},  a_{0}  \}$ a {\em virtual cycle}, or if we want to emphasize the period, {\em virtual cycle of period $n$}.
\end{definition}

 This definition is justified by the following.  Assume  for argument's sake that we are in the first case.  Let $\gamma(t)$ be an asymptotic path for $\la=a_1$, that is, $\lim_{t\to \infty} \gamma (t) =\infty$ and $\lim_{t\to \infty} f(\gamma(t))=a_1$.

Then $\hat\gamma=f^{-n}(\gamma(t))$ is again an asymptotic path where the inverse branches are chosen so that $f^{-1}(a_i)=a_{i-1}$; that is,  \[ \lim_{t \to \infty} f(\gamma(t))=\lim_{t \to \infty} f^{n+1}(\hat\gamma(t))=\la=a_1,  \]
so that in this limiting sense, the points form a cycle.

\begin{definition} Let $\Omega_n$ be a shell component of period $n$ and let
\[\mathbf a_\la=\{a_0, a_1, \ldots, a_{n-2},  a_{n-1} \} \]
be the attracting cycle of period $n$ that attracts $\la$ or $\mu$.   Suppose that as $k \to \infty$,  $\la_k \to \la^* \in \partial{\Omega_n}$ and     the multiplier $\nu_{\la_k}=\nu(\mathbf a_{\la_k})=\Pi_{i=0}^{n-1} f'(a_{i}(\la_k)) \to 0.$   Then $\la^*$ is called a {\em virtual center} of $\Omega_n$.
\end{definition}

\begin{remark} Note that if $n=1$ and $a_0(\la_k)$ is the fixed point,  the definition implies that $f'(a_{0}(\la_k)) \to 0$.  This in turn implies that either $\la$ tends to $\infty$ or $\la$ tends to the parameter singularity $\rho/2$ so that $\mu$ tends to $\infty$.   These ``would be'' virtual centers do not belong to the parameter space but they share many properties with proper virtual centers including transversality (see definition~\ref{trans}).
\end{remark}

Since the attracting basin of the cycle $\mathbf a_\la$ must contain an asymptotic value, we will assume throughout the paper that the points in the cycle are labeled so that $\la$ or $\mu$ and $a_1$ are in the same component of the immediate basin.

In the next theorem we collect the results in \cite{FK} about shell components  for fairly general families of functions.  The proof of Parts [b] and [c] are based on an estimate of the growth of the orbits of the singular values given in  lemma 2.2 of \cite{RS},  and on proposition 6.8 of \cite{FK}. Part [d] is theorem 6.10 of \cite{FK}. Part [e] combines the accessibility in Part [d] and theorem A of \cite{CK}, whose proof contains a construction that shows that every virtual cycle parameter is on the boundary of a shell component.

\begin{thm}[Properties of Shell Components of $\calf_{2}$]\label{thm:shell components}
Let $\Omega$ be a shell component in $\calf_{2}$.  Then
\begin{enumerate}[(\rm a)]
\item The map $\nu_\la: \Omega \rightarrow \DD^*$ is a universal covering map.  It extends continuously to $\partial\Omega$ and $\partial\Omega$ is piecewise analytic;  $\Omega$ is  simply connected and   $\nu_\la$ is infinite to one.

\item There is a unique virtual center on $\partial\Omega$.   If   the period of the component is $1$ and $\Omega$ is a shell component of $\calm_{\la}$, the component is unbounded and the virtual center is at infinity;  if, however,  $\Omega$ is a shell component of $\calm_{\mu}$ of period $1$, then it is bounded and the virtual center is at the finite point $\rho/2$ which is a parameter singularity.  This is the only difference between $\calm_{\la}$ and $\calm_{\mu}$.  

\item  If $\la_k \in \Omega$ of period greater than $1$ is a sequence tending to the virtual center  $\la^*$ and $a_0(\la_k)$ is the periodic point of the cycle $\mathbf a(\la_k)$ in the component containing the asymptotic tract and $a_1(\la_k)=f_{\la_k}(a_0(\la_k))$, then as $k \to \infty$, $a_0(\la_k) \to \infty$ and $a_1(\la_k) \to \la^*$.

\item  Every   virtual center of a shell component is a virtual cycle parameter and it is an accessible boundary point.

\item  Every virtual cycle parameter is a virtual center.
\end{enumerate}
\end{thm}

As a corollary we have
\begin{cor}\label{pre-pole centers} If $\la^*$ is a virtual center of a shell component of period $n$, there are virtual centers of  period $n+1$ accumulating on it.
\end{cor}
\begin{proof}  The poles of $f_{\la}$ are given by 
\begin{equation}\label{poles}
 p_k(\lambda)=\frac{1}{2}\Log (\frac{\rho-2\lambda}{\rho})+ik\pi, 
\end{equation}
 where $\Log$ is the branch of the logarithm with imaginary part in $[-\pi,\pi)$.  They and all their preimages are holomorphic functions of $\la$.

Let $V$ be a neighborhood of $\la^*$ in the parameter plane that does not contain any poles of $f_{\la}^{k}$ for all $1\leq k<n-1$.   Such a neighborhood exists because the poles of $f_{\la}^{k}$ form a discrete set.  The holomorphic function $h(\la)=f_{\la}^{n-1}(\la)$ maps $V$ to a neighborhood $W$ of infinity and $h(\la^*)=\infty$.  Since infinity is an essential singularity,  for $\la \in V$  and large enough $|k|$, $W$ contains infinitely many poles $p_k(\la)$ of the functions $f_{\la}(z)$;  moreover, for each $\la$, as $k \to \infty$,    $p_k(\la)$ converge to infinity. The zeroes of the functions $h_k(\la)=h(\la)-p_k(\la)$  are virtual centers of components of period $n+1$.   We want to show there is a sequence of these zeroes in $V$ converging to $\la^*$.

The functions $\hat{h}(\la)=1/h(\la)$ and $\hat{p}_k(\la)=1/p_k(\la)$ take values in a neighborhood of the origin.    Since the $p_k(\la)$ converge to infinity as $|k| \to \infty$  uniformly on $\overline{V}$ as long as $V$ is small enough,  we can find $N$ large enough so that if $k>N$ and $\la \in \partial V$, then  $|\hat{p}_k(\la)| <  |\hat{h}(\la)|$.  By Rouch\'e's theorem, we conclude $\hat{h}$ and $\hat{h} -\hat{p}_k$ have the same number of zeroes in $V$;  $\hat{h}$ has a zero at $\la^*$ 
and thus  each $\hat{h} -\hat{p}_k$ has a zero $\la_k \in V$;  it follows that  $h(\la_k) =p_k(\la_k)$ so that $\la_k$ is a virtual center of period $n+1$.
\end{proof}

\subsection{Combinatorics}\label{combinatorics}

Theorem~\ref{thm:shell components} allows us to assign a  label to each of the shell components   
 of $\calm_{\la}$ in terms of its virtual center.   To label the virtual centers we need to know that the indices of the poles are well defined.   In section~\ref{Inverse branches} we will prove lemma~\ref{invwelldef}  that says that we can find a simply connected domain $\Sigma$, containing  $\calm_{\la}$ and not containing $\calm_{\mu}$, in which,  after an initial choice, as above, of a basepoint and  a branch of the logarithm,    the poles and inverse branches of $f_{\la}$ can be labelled consistently.    The discussion here will assume that lemma.  
 
Pick a basepoint   that   is not in $\calm_{\la}$, for example,   the symmetric point $\la_0 =-\mu_0=\rho$. 
It has poles  $p_k(\la_0)$ defined by the principal branch of the logarithm.   
With the poles $p_k(\la_0)$ defined by equation~(\ref{poles}), denote the branch of $f_{\la_0}^{-1}$ that maps $\infty$ to $p_k(\la_0)$  by  $g_{\la,k}(z)=g_{\la_0,k}(z)$.   With the  choice of a fixed base point and logarithm branch,  the inverse branches are well defined since the set $\Sigma$ (to be defined in section~\ref{Inverse branches})  is simply connected.  

We use these branches to define labels for the prepoles of all orders, and thus for labels of the virtual cycle parameters.  Because of part [d] of theorem~\ref{thm:shell components} each virtual cycle parameter is a virtual center of a shell component so the label of the virtual center  defines a label for the shell component.    
 
 The formula for  $p_{k}(\la)$ shows that the poles are injective functions of  $\la$ in $\Sigma$.   Let 
 \[ \calv_1=\{ \la^*_k \in \Sigma \, | \,  g_{\la^*,k}(\infty)=  \la^* \}. \]
     That is, $\calv_1$ is the set of $\la^*_k$ such that $f_{\la^*_k}(\la^*_k)=\infty $. It  is 
 the set of  virtual cycle parameters of order $1$ and hence virtual centers of shell components $\Omega_2$ of period $2$.   We assign the label $k$ to each point in $\calv_1$ and the same label to  the component for which it is the virtual center.   
 
 The prepoles $p_{k_1k_2}(\la)=g_{\la,k_2}(p_{k_1}(\la))$ are defined for all $\la \in \Sigma \setminus \calv_1$.  Since they  are holomorphic functions of $\la$ with non-vanishing derivative,  each $g_{\lambda, k}$ is an injective function of both  $\lambda$ and $z$.  Next, we inductively define the sets of virtual cycle parameters of order $n-1$ with labelled points by 
  \[ \calv_{n-1}=\{ \la^*_{k_{n-1} \ldots k_1} \in \Sigma \setminus \cup_{i=1}^{n-2} \calv_i \, | \,  g_{\la^*_{k_{n-1} \ldots k_1},k_{n-1} \ldots k_1}(\infty)=  \la^*_{k_{n-1} \ldots k_1}\}. \]
   The prepoles $p_{k_{n-1} \ldots k_1}(\la)$ are defined for all $\la \not\in \calv_{n-1}$ and, as above, move injectively.  
   
     We now assign the label  $k_{n-1} \ldots k_1$ to the shell component of order $n$ for which $\la^*_{k_{n-1} \ldots k_1}$ is the virtual center.  
    
\begin{definition}\label{itin}    We call the label $\mathbf k_n=k_nk_{n-1} \ldots k_1$ assigned to each prepole and each virtual cycle parameter its {\em itinerary}.
\end{definition}

We can also use the labelling of the inverse branches to assign an itinerary to each attractive cycle. 
 \begin{definition}\label{itin cycle}
    For simplicity we suppress the dependence on $\la$ and assume the shell component is in $\calm_{\la}$.  Suppose $f^n(a_0)=a_0$ for $n \geq 1$,  where, by our numbering convention in part [c] of theorem~\ref{thm:shell components}, $a_0$ is in the asymptotic tract of $\la$ and $a_j=f(a_{j-1})$, $j=1, \ldots n$.  Then for some $k_j$, $a_{j-1}= g_{k_j}(a_{j})$.   In fact there is  a unique sequence $\{k_1,\ldots, k_{n}\}$ such that
 \[  a_0= g_{k_{n}} \circ \ldots g_{k_2}   \circ g_{k_{1}}( a_0).\]
We say  the cycle $\mathbf a$ has {\em itinerary}
$ \mathbf k_n=  k_{n}k_{n-1}\ldots k_2 k_1$.
\end{definition}
 
\begin{prop}\label{compitin} Let $\Omega_n$ be a shell component and suppose for $\la_0 \in \Omega_n$, the cycle $\mathbf a(\la_0)$ has itinerary $k_nk_{n-1} \ldots k_1$.   Then  for every $\la \in \Omega_n$, the itinerary of $\mathbf a(\la)$ is of the form $k_{0,j} k_{n-1} \ldots k_1$ for some $j \in \ZZ$. 
\end{prop}

\begin{proof} If the component of the basin $\mathbf a(\la)$ containing $a_j(\la)$ is denoted by $D_j(\la)$, then for  $j=1, \ldots n-1$,  $f_{\la}:D_j(\la) \rightarrow D_{j+1}(\la)$ is one to one.  Inside  $\Omega_n$, the points of the periodic  cycle  move holomorphically and are related by the  inverse branches $g_{\la,k_j} : D_{j+1}(\la) \rightarrow D_j(\la)$.  The branch is the same  for all $\la  \in \Omega_n$ since it  is simply connected;  in it the $g_{\la,k_j}$ are quasiconformally conjugate and the $a_j(\la)$ move holomorphically.  At the last step in the cycle,  however,  the map $f_{\la}:D_0(\la) \rightarrow D_{1}(\la)$   is infinite to one and so $a_1(\la)$ has infinitely many inverses, $a_{0,j}(\la) \in D_0(\la)$.  They are all in the asymptotic tract of $\la$ but only one of them can belong to the cycle. Thus although the inverse branch $g_j=g_{0,j}$ is well defined for each $\la$, the branch that defines the cycle  changes as $\la$ moves in $\Omega_n$.  
\end{proof}

Above we assigned a label, or itinerary to the virtual center of each shell component.   We now address the questions of the uniqueness of these labels and their relation to the itineraries of their attracting cycles.  As we stated above, this is based on lemma~\ref{invwelldef}, which will be proved later, that the inverse branches are defined as single valued functions of $\la$.

\begin{prop}\label{label shells} Every shell component $\Omega_n \in \calm_{\la}$ and $\Omega_n' \in \calm_{\mu}$, $n>1$,  has a unique label defined  by the itinerary of its virtual center $\la^*$, a pre-pole of order $n-1$  where $n$ is the minimal such integer.
\end{prop}

Because the shell components of period $1$ have virtual centers that do not belong to the parameter space, we cannot label them in this way.  There are only two such points, $\rho/2$ and $\infty$ and hence only two such components with no label.  In order to have a label for every component, we  arbitrarily assign the label $\infty$ to these components.
\begin{proof}
The boundary of each  shell component $\Omega_n$ contains one and only one virtual center $\la^*$ \footnote{It will follow from the Common Boundary Theorem that it is on the boundary of only one shell component.  This is different from the tangent family where pairs of shell components share virtual centers. See e.g. \cite{CJK19}} and the label  $ \mathbf k_{n-1}=  k_{n-1}k_{n-2}\ldots k_1$ of the virtual center is its itinerary.       
Let $V$ be  a neighborhood of $\la^*$ and let $W=\Omega_n \cap V$.  By proposition~\ref{compitin}, the itineraries of the points in $W$ agree except for their first entry.  By proposition 6.8 of \cite{FK}, as $\la \in W$ tends to the virtual center,  the point $a_0(\la) = g_j(a_1(\la))$ of the cycle tends to infinity and the point $a_1(\la)$ tends to the virtual cycle parameter $\la^*$, a pre-pole of order $n-1$ with 
itinerary, $ \mathbf k_{n-1}=  k_{n-1}k_{n-2}\ldots k_1$.

Note that because $\lambda \in \mathcal{M}_\lambda$ or $\mathcal{M}_\mu$,  the cycle $\mathbf a(\la)$ attracts only one of     the asymptotic asymptotic values. Therefore unlike the tangent family, where both asymptotic values can be attracted by a single cycle of double the period, $n-1$ is minimal.
\end{proof}

The proof of Corollary~\ref{pre-pole centers} also implies that
\begin{prop}\label{labels acc} Let $\la^*$ be the virtual center of a shell component  $\Omega_n$ and let $\Omega_{n+1,i}$, be a sequence of components whose virtual  centers $\la^*_i$   converge to $\la^*$ as $i$ goes to infinity.   If the itinerary of $\la^*$  is given by $  \mathbf k_{n-1}=  k_{n-1}k_{n-2}\ldots k_1 $, the itineraries of the $\la^*_i$ are given by $\mathbf k_{n,i}= k_{n-1}k_{n-2}\ldots k_1 k_{0,i} $.
\end{prop}

\begin{remark}    There is an interesting duality here.   As we approach the virtual center from inside a shell component of order $n$, we are taking a limit of cycle itineraries; the first entry in the itinerary (corresponding to the last inverse branch applied) disappears. Thus an itinerary with $n$ entries becomes one with $n-1$ entries.  However, if we consider the labels of the shell components of order $n+1$ approaching the shell component of order $n$,  it is the last entry (corresponding to the first inverse branch applied) that disappears in the limit.
 \end{remark}

 \begin{proof} 
As above, 
 let $V$ be a small neighborhood of  $\la^*$.  We may assume it contains no virtual center of order less than $n-1$.  The functions  $g_{\la,k_j}$ that define the virtual cycle $a_j(\la^*)$, $j=1, \ldots n-1$, are defined in  $V \cap \Omega_n$ where they track the attracting cycle.  They   also extend to all of $V \setminus \{\la^*\}$ by analytic continuation.  Also for  $\la \in V \cap \Omega_n$,  the functions $a_{0,i}(\la)=g_{\la,k_{0,i}}(a_1(\la))$ are defined for all $i$ but for only one $i$ does it belong to the attracting cycle.  All of these functions extend to $V \setminus \{\la^*\}$.    
 
  Now let $W$ be a neighborhood of infinity and 
 let $G(\la,z)$ be a map from $V \times W$ to $\CC$ defined by $g_{\la,k_{n-1}} \circ \ldots \circ g_{\la,k_1}(z)$.   By corollary~\ref{pre-pole centers} 
 the neighborhood $W$ contains the virtual centers $\la_i^*$ of a sequence of shell components $\Omega_{n+1,i}$ with limit $\la^*$.  These are  poles $p_i^*$ of $f_{\la_i^*}^{n-1}$ so we can find inverse branches  of $f_{\la}$, which we denote by $g_{\la,k_{0,i}}$, such that $p_i^*=g_{\la^*_i,k_{0,i}}(\infty)$.
 It then follows that the itineraries of the $\la_i^*$ are $\mathbf k_{n,i}=  k_{n-1}k_{n-2}\ldots k_1 k_{0,i}$ as claimed.
\end{proof}

 Thus  the combinatorics of the prepoles enable us to   label   each shell component $\Omega_n \in \calm_{\la}$ and $\Omega_n' \in \calm_{\mu}$ by the itinerary of its virtual center. 
   If $  \mathbf k_{n-1}= k_{n-1}k_{n-2}\ldots k_1 $ is the itinerary of the virtual center of a shell component of period $n$, and we want to emphasize it, we write $\Omega_{\mathbf k_{n-1}}$ or $\Omega_{\mathbf k_{n-1}}'$.   

The above discussion, modulo the proof of lemma~\ref{invwelldef},  gives us a proof of   the Combinatorial Structure Theorem:

\medskip
\begin{starthm}[Combinatorial Structure Theorem]~\label{main2}
 The virtual cycle parameters $\lambda_{{\mathbf k}_{n}}$ of order $n$ can be labelled by  sequences ${\mathbf k}_n=k_n k_{n-1} \ldots k_1$,
where $k_i \in \ZZ$, in such a way that each of the parameters $\lambda_{{\mathbf k}_{n}}$ is an accumulation point in $\CC$
of a sequence of parameters $\lambda_{{\mathbf k}_{n+1}} $ of order $n+1$, where ${\mathbf  k}_{n+1} =k_{n}k_{n-1} \ldots k_1k_{0,j} $, $j\in \ZZ$.
This combinatorial description  of the virtual cycle parameters determines combinatorial descriptions of the sets $\calm_{\la}$ and $\calm_{\mu}$. 
\end{starthm}

  \subsection{Parameter space pictures}
   \begin{figure}
     \centering
  \includegraphics[width=5in]{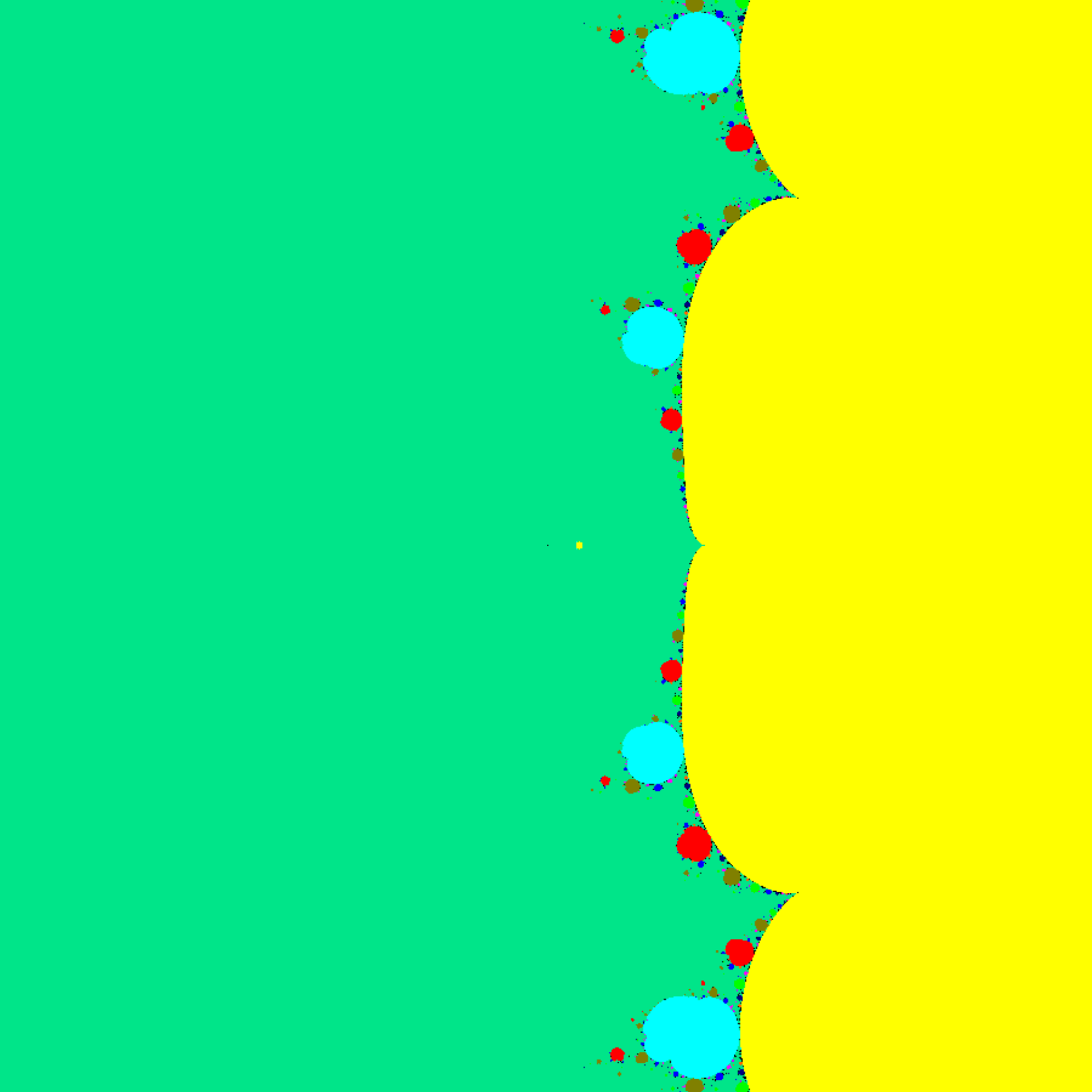}
  \caption{The $\la$ plane divided into the shift locus and shell components.  The green region represents the shift locus $\cals$. The regions $\calm_{\la}$ and $\calm_{\mu}$ are colored by the period of the component: period $1$ is yellow , period $2$ is cyan, period $3$ is red, etc.  The coloring is not visible for $\calm_{\mu}$ because it is so small. }
  \label{lambdaplane}
\end{figure}

   \begin{figure}
  \centering
  \includegraphics[width=5in]{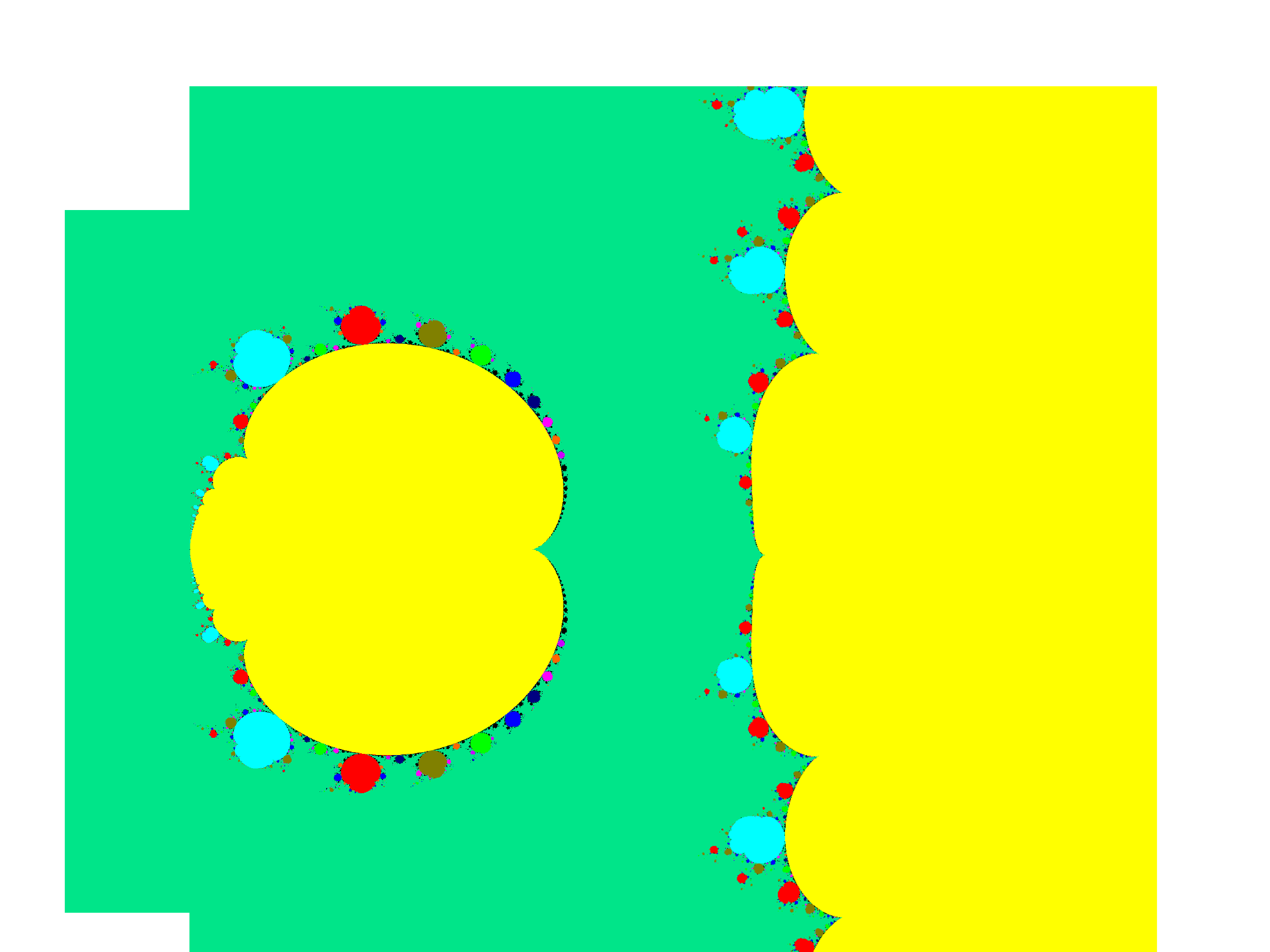}
  \caption{Blow up of $\calm_{\mu}$ placed near $\calm_{\la}$ for comparison. The coloring scheme is the same as in figure~\ref{lambdaplane} and is now visible in   the  blown up $\calm_{\mu}$.}\label{calm}
\end{figure}

\begin{figure}
 \centering
  \includegraphics[width=4in]{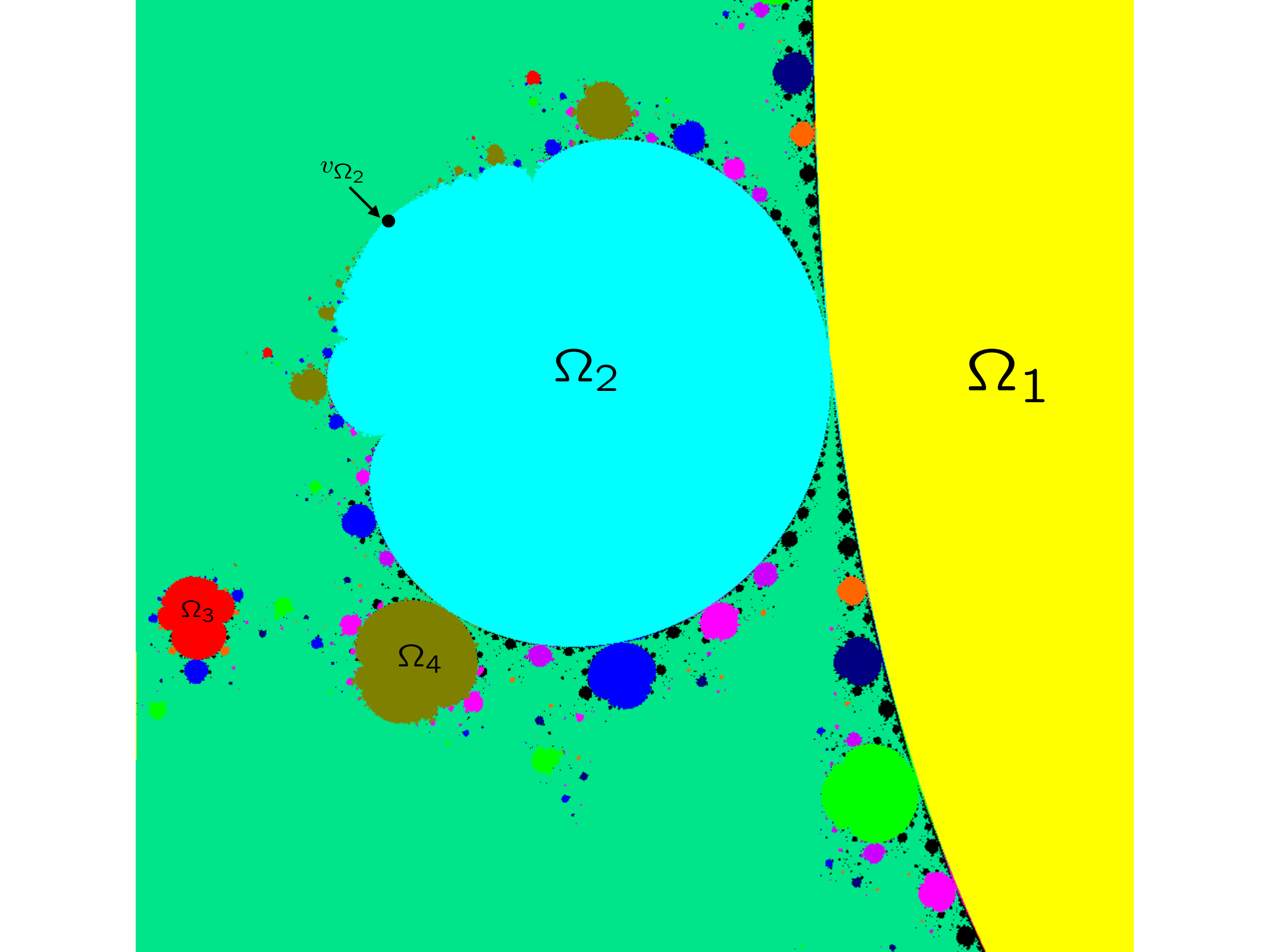}
  \caption{ Blow-up of the $\la$ plane near $\calm_{\la}$ with the periods labelled. }\label{blowup}
\end{figure}

Figure~\ref{lambdaplane} shows a picture of the $\la$ parameter plane for $\rho=2/3. $
   The green region is $\cals$ where both $\la$ and $\mu$ are attracted to the origin.
The unbounded multicolored region on the right in figure~\ref{lambdaplane}  is $\calm_{\la}$ and the  small bounded multicolored region inside the green region is $\calm_{\mu}$.   Since this figure is drawn to scale, in $\calm_{\mu}$ the colors other than yellow are not visible. To make the structure of this region visible and show it is  similar to $\calm_{\la}$'s,  in figure~\ref{calm} we place a blown up  neighborhood of $\calm_{\mu}$    near $\calm_{\la}$.  

     The shell components are colored according to their period:  yellow is period $1$,  cyan is period $2$, red is period $3$ and so on.   Periods higher than $10$ are colored black.  Note that there is only one unbounded domain, the yellow period $1$  domain on the right, $\Omega_1$;  its virtual center is the point at infinity.  The virtual center of the period $1$ component of $\calm_{\mu}$ is the leftmost point.  It is the singular point $\rho/2$ of the parameter space. There is a cusp  boundary point of  $\Omega_1$ on the real axis where the multiplier of the cycle attracted to $\la$ is $+1$. There are cyan period $2$ components appearing as ``buds'' off of the yellow component $\Omega_1$ where  the  multiplier of the cycle attracted to $\la$ is $e^{(2m+1) \pi  i}$; each of these has a virtual center with itinerary $\mathbf k_1=m$.

  In figure~\ref{blowup}, we see a period $2$ component  $\Omega_2$ budding off $\Omega_1$.   Although  $M_{\la}$ and $M_{\mu}$ look disconnected in the figure, as we will prove,  they are not.  Here we have only computed shell components for periods up to $10$.   To make a figure where $\cals$ and $\calm_{\la}$  look connected would require much more computation and many more colors to show components with much higher periods.   What we do see, however, is a period $3$, red component that is NOT a bud component of the period $1$ component.  In fact, there are infinitely many such converging to the virtual center of $\Omega_2$ marked as $v$.  We postpone a full discussion of the finer structure of the shell components to future work.

\section{Boundaries of Hyperbolic Components and Virtual Cycle Parameters}
\label{trans}

  In this section we  show that each virtual center is a boundary point of both $\cals$ and either $\calm_{\la}$  or $\calm_{\mu}$.  To do this we  need to use the concept of transversality.

\medskip
   \begin{definition}[Transversality,~\cite{ CJK19}]~\label{trans}   Suppose $\la^{*}$ is a virtual cycle parameter.  Let  $p^*(\la)$ be the holomorphic prepole function   such that $p^*(\la^{*}) =f_{\lambda^*}^{n-2}(\lambda^*)$.  Define the holomorphic function,
   \[   c_n(\la)=f_{\lambda}^{n-2}(\lambda)- p^*(\la).  \]
   We say $f_{\la}$ is {\em transversal at $\la^{*}$} or satisfies a  {\em transversality condition} at $\la^*$ if $c_{n}'(\la^{*}) \not=0$.
       \end{definition}

\medskip
\begin{thm}[Common Boundary Theorem]\label{vc on bdy}
Every virtual cycle parameter is a boundary point of both a shell component and the shift locus.
Furthermore,  the family $\{f_{\la}\}$ is transversal at these parameters.
\end{thm}

\medskip
\begin{remark}~\label{tranrmk} The transversality property translates to the dynamic planes of the functions $f_{\la}$ as follows:\\
 If  $ f_{\la}$ is transversal at $\la^{*}$,  and if  $\la(t)$ is any smooth path passing through $\la^{*}$ at $t^{*}$ such that $\la' (t^{*})\not=0$, then the dynamics of $f_{\la(t)}$ bifurcates at $t^{*}$.   In particular, as $\la(t)$ moves from a shell component into the shift locus through the common boundary point, an asymptotic value, say $\la(t)$,  moves from the attracting basin of an attractive cycle of $f_{\la(t)}$, through the pre-pole $\la^*$ of the virtual cycle, and into the attracting basin of zero for $f_{\la(t)}$. Moreover,  if  $\epsilon$ is small enough so that $\la(t)$ does not contain any other  virtual center when  $|t-t^*|<\epsilon$, then $t^*$ is the only point in the interval  $|t-t^*|<\epsilon$ where the dynamics of $f_{\la(t)}$ bifurcate.  This is illustrated in  figures \ref{param} and~\ref{dynam}.
 
 In addition, transversality of $f_{\la}$   at $\la^{*}$ implies that  the holomorphic functions defining the poles $p_k(\la)$ and pre-poles  $p_{\mathbf k_n}(\la)$  satisfy  $p_k'(\la^*) \neq 0$ and $p_{\mathbf k_n}'(\la^*) \neq 0.$
  \end{remark}

{\it Proof of the Common Boundary Theorem.} Let $\la^*$ be a virtual cycle parameter.  It follows from part (e) of theorem~\ref{thm:shell components},   that   $\la^{*}$ is on the boundary of a shell component.  Suppose this component,  $\Omega_{n}$, is in $\calm_{\la}$ so that $\mu^{*}$ is in $A_{\la^{*}}$, the attracting basin of $0$,   and $f_{\la^{*}}^{n-1} (\la^{*})=\infty$.  We can choose  $U$ be a small neighborhood  of $\la^{*}$ such that $\cap_{\la \in U} A_{\la} $ contains $\mu(\la)$  for all $\la \in U$ and  $a_{n-1}(\la)= f_{\la}^{n-1} (\la)$ is a holomorphic function on  $U$  with  $a_{n-1}(\la^{*})=\infty$.  Since infinity is always a boundary point of the basin $A_{\la}$,     the open mapping theorem  implies that there is a $\la_{U}\in U$ with  $\la_{U}\in A_{\la(U)}$. This says $\la_{U} \in {\mathcal S}$ and thus $\la^*$ is a boundary point of $\cals$.

 In~\cite{CJK19}, we proved a transversality theorem for maps in the tangent family, $\la \tan z$ with $\la = it,$ $t \in \RR$.  There $\la(t)$ is in the imaginary axis, and the proof shows  that the function $c_n(\la(t))$  has no critical point at $t^*$.   It involves the use of holomorphic motions and some ideas adapted from~\cite{LSS}.       That proof can be adapted here by replacing the imaginary axis  with a path $\la(t)$ in $\Omega_n$ defined by the condition that  the multiplier of the attracting cycle  $\mathbf a(\la)$ has argument  equal to  $2\pi i n$, for some $n$.  Then the arguments there can be applied and show that as   $t \to t^*$ in $\Omega_n$, $c_n'(\la(t^*) \neq 0$, and the dynamics bifurcates smoothly. 
We refer the interested reader to that paper for the details. \hspace{4.1in} $\Box$

 An immediate corollary of the Common Boundary Theorem is

 \begin{cor} Given an itinerary, $ \mathbf k_{n-1}=  k_{n-1}k_{n-2}\ldots k_1$, there is exactly one component in each of $\calm_{\la}$ and $\calm_{\mu}$ with that itinerary label.
 \end{cor}

 \begin{proof}  Let  $\mathbf k_{n-1}$ be a given itinerary.     The prepoles $p_{\mathbf k_{n-1}}(\la)$ of order $n-1$ form a discrete set in dynamic space since they are solutions of $f^{n-1}_\lambda(z)=\infty$ and   there is only one with itinerary $\mathbf k_{n-1}$.  They are on the boundary of $A_{\la}$. 
 The virtual centers form a discrete set in parameter space since they are solutions of $f^{n-1}_\lambda(\la)=\infty$.  
 
 We can find a sequence $\la_j \in \cals$, tending to $\partial\cals$ as $j$ goes to infinity, such that  $|f_{\la_j}^{n-1}(\la_j) -p_{\mathbf k_{n-1}}(\la)|$ goes to zero as  $j$ goes to infinity.  It follows that $\lim_{j \to \infty} \la_j$ is a virtual cycle parameter $\la^*$ with
itinerary $\mathbf k_{n-1}$.    By theorem~\ref{trans}, in a small neighborhood of $\la^*$, there is no other virtual center with itinerary $\mathbf k_{n-1}$ so that the component $\Omega_n$ with $\la^*$ as virtual center is the only one in $\calm_{\la}$ with this itinerary. 

We obtain a different component $\Omega_n'$ if we choose a sequence $\la_j'$ such that $f_{\la_j}^{n-1}(\mu_n)$ approaches the prepole with this itinerary, but that is the only other possibility.  In this case, $\Omega_n'$ is in $\calm_{\mu}$.
   \end{proof}

  \begin{figure}
  \centering
  \includegraphics[width=2in]{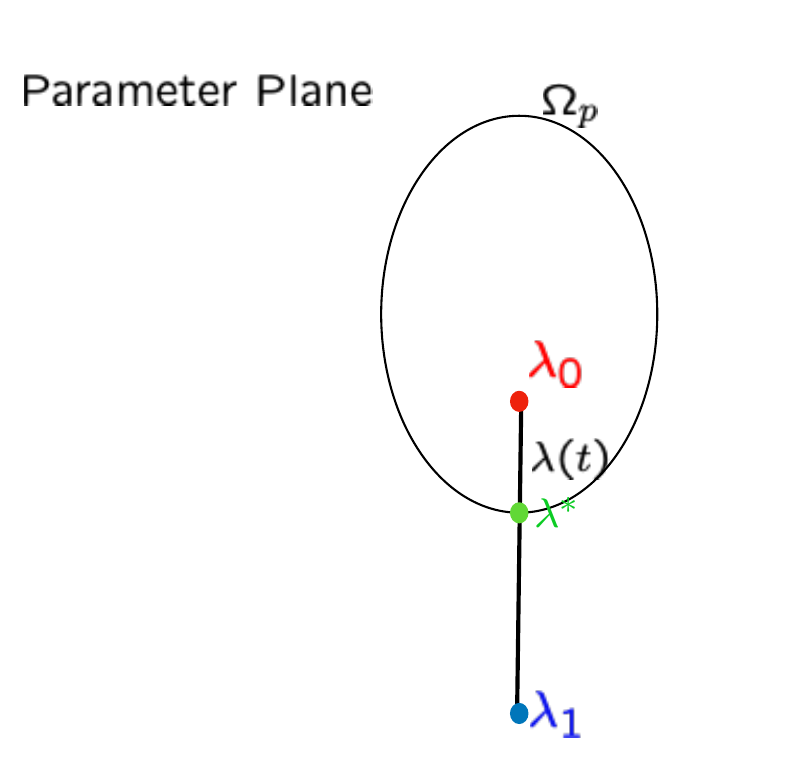}
  \caption{Transversality in the parameter plane} \label{param}
\end{figure}

\begin{figure}
  \centering
  \includegraphics[width=5in]{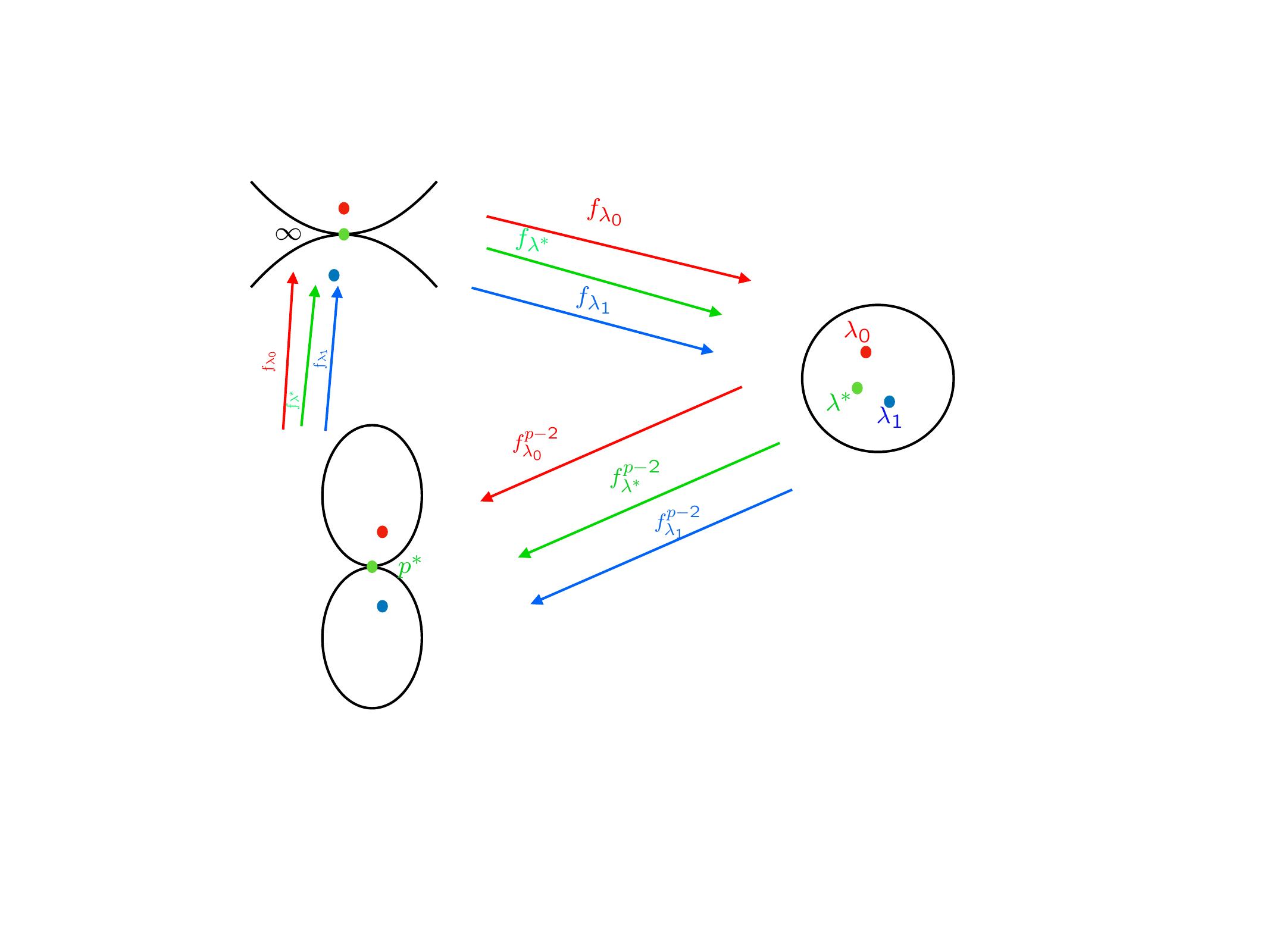}
  \caption{Transversality in the dynamic plane}\label{dynam}
\end{figure}

 \subsection{J-stability and the Bifurcation Locus}\label{biflocus}
Denote the set of virtual cycle parameters by $\mathcal{B}_{cv}$.  By theorem~\ref{thm:shell components}, each such parameter is on the boundary of a unique shell component and in section~\ref{combinatorics} we used these parameters to enumerate the shell components.    Here we will prove that these parameters  are dense in the boundary of the shift locus.  To do so we need two definitions.

  \begin{definition}[Holomorphic family]
A {\em holomorphic family} of meromorphic maps over a complex manifold $X$  is a  map
$\calf:X \times \CC \rightarrow \hat\CC$, such that $\calf(x,z)=:f_x(z)$ is meromorphic for all $x\in X$ and $x \mapsto f_x(z)$ is holomorphic for all $z\in \CC$.
\end{definition}

\begin{definition}
The  \emph{J-stable set} of the family $\mathcal{F}_{2}$, denoted by $\mathcal{J}=\mathcal{J}_{\rho}$,  is the set $\{\lambda  \ |\ f^n_{\lambda}(\lambda)$  and $f^n_{\lambda}(\mu)$ are well defined in a neighborhood about $\lambda$ for all $n$ and form normal families $\}$. Its complement is  called the \em{bifurcation locus}.
\end{definition}

Theorem B of \cite{MSS} in our context states 
\begin{thm}\label{MSS} In any holomorphic family of meromorphic maps with finite singular set,  $\mathcal{J}$ coincides with the set of parameters for which the total number of attracting and superattracting cycles of $f_{\la}$ is constant on a neighborhood of $\la$.
\end{thm} 

 As a corollary (see ~\cite{KK}, Corollary 3.2)
\begin{prop}\label{locconst} If $\la_0 \in \mathcal J$, then  the number of attracting cycles of $f_{\la_0}$ is locally constant in a neighborhood of $\la_0$; in particular, $\mathcal J$ is open.
\end{prop}

Consider a component $U$ of ${\mathcal J}$. Suppose $\la$ and $\la'$ are two points in $U$ and $\gamma (t): [0, 1]\to U$ is an analytic curve connecting $\la$ and $\la'$ with $\gamma(0)=\la$ and $\gamma(1)=\la'$. In a neighborhood $V$ of $\gamma$ with  basepoint $\la$,
the set $E_{c}=\{p_{c}\}$, $c\in V$, of all repelling periodic points of $f_{c}$ defines a holomorphic motion,
 $$
h(z, c) =p_{c}: E_{\la} \times V\to \chat.
$$
The $\lambda$-lemma (see \cite{MSS,GJW}) implies that the holomorphic motion $h$ can be extended to the closure $\overline{E}_{\la}$ of $E_{\la}$; that is, there is a holomorphic motion 
$$
H(z,c): \overline{E}_{\la} \times V\to \chat
$$ 
such that $H|E\times V=h$ and $H_{\la, \la'} (z)=H(z, \la'): \overline{E}_{\la} \to \overline{E}_{\la'}$ is a quasiconformal homeomorphism. Since the repelling periodic points of $f_{\la}$ are dense in the Julia set $J_{\la}$, it follows that $\overline{E}_{\la} = J_{\la}$ and 
$$
H_{\la, \la'} \circ f_{\la} =f_{\la'}\circ H_{\la, \la'} \quad \hbox{on $J_{\la}$}. 
$$ 
Note that the construction of $H_{\la, \la'}$ depends on the choice of the curve $\gamma$ and it may not be unique.

%{\b Consider a component $U$ of ${\mathcal J}$ with a basepoint $\la_{0}$. Let $E_{\la}=\{p_{\la}\} $ be the set of all repelling periodic points of $f_{\la}$ and set $E=E_{\la_{0}}$. Then
%$$
%h(z, \la) =p_{\la}: E\times U\to \chat
%$$
%defines a holomorphic motion. The $\lambda$-lemma (see \cite{MSS,GJW}) implies that the holomorphic motion $h$ can be extended to the closure $\overline{E}$ of $E$; that is, there is a holomorphic motion 
%$$
%H(z,\la): \overline{E} \times U\to \chat
%$$ 
%such that $H|E\times U=h$ and $H_{\la} (z)=H(\la, z): \overline{E} \to \overline{E}_{\la}$ is a quasiconformal homeomorphism. Since the repelling periodic points of $f_{\la}$ are dense in the Julia set $J_{\la}$, it follows that $\overline{E}_{\la} = J_{\la}$. For any $\la, \la'\in U$, let $$
%H_{\la, \la'}= H_{\la'}\circ H_{\la}^{-1}: \overline{E}_{\la}\to \overline{E}_{\la'}.
%$$ 
%It is a quasiconformal conjugacy between $f_{\la}$ and $f_{\la'}$ on their Julia sets, that is; 
%$$
%H_{\la, \la'} \circ f_{\la} =f_{\la'}\circ H_{\la, \la'} \quad \hbox{on $J_{\la}$}. 
%$$ 
%}
 
We also need the following generation of Montel's theorem.

\begin{thm}[See theorem $3.3.6$  in \cite{Bea}]\label{generalMontel}
Let $D$ be a domain, and suppose that the functions $\phi_1$, $\phi_2$ and $\phi_3$ are analytic in $D$, and are such that the closures of the domains $\phi_j(D)$ are
mutually disjoint. If  $\mathcal{F}$ is a family of functions, each analytic in $D$, and such
that for every $z$ in $D$, and every $f$ in $\mathcal{F}$, $f(z)\neq \phi_j(z)$, $j=1, 2, 3$,  then $\mathcal{F}$ is
normal in $D$.
\end{thm}

The set of virtual center parameters $\mathcal{B}_{cv}$ is
clearly not contained in $\mathcal{J}$.  By theorem~\ref{thm:shell components} and theorem~\ref{vc on bdy} the points in $\mathcal{B}_{cv}$ are on the boundaries of both a shell component and the shift locus.  In addtion, we have

\begin{thm}\label{bndry}
The boundary of $\mathcal{J}$ is contained in the closure of $\mathcal{B}_{cv}$, that is, $\partial \mathcal{J}\subset \overline{\mathcal{B}_{cv}}$.
\end{thm}

\begin{proof} Since $0$ is an attracting fixed point, at least one of the families  $\{f_{\lambda}^n(\lambda)\}$ and $\{f_{\lambda}^n(\mu)\}$ converges to $0$; that is, for each $\la_0$, one of them is always normal and, by proposition~\ref{locconst},  in a neighborhood of $\la_0$ it is   the same family that  is normal.   Suppose $\lambda_0\in \partial \mathcal{J}$;  without loss of generality, we may assume that $\{f_{\lambda}^n(\lambda)\}$ is not normal at $\lambda_0$. 

Let $U$ be any neighborhood of $\lambda_0$. The poles of $f_{\lambda,}$
\[  p_k(\lambda)=\frac{1}{2}\log (\frac{\rho-2\lambda}{\rho})+ik\pi, \, k\in \mathbb{Z}, \]
form a holomorphic family in $U$.
If $f^n_{\lambda,\mu}(\lambda) \neq p_k(\lambda)$ for any $k$ or $\la \in U$, then theorem \ref{generalMontel} implies  $f^n_{\lambda}(\lambda)$ is normal in $U$. This contradicts the hypothesis that $\lambda_0\in \partial \mathcal{J}$.
 \end{proof}

The parabolic cusps and Misiurewicz points are contained in the bifurcation locus.

\section{Topological structure  of the Shift Locus}
\label{Shift}

In this section we will show that the shift locus is homeomorphic to an annulus punctured at one point.  This puncture corresponds to the point $\la=0$ where $f_{\la}$ is not defined.

Before we discuss this proof we need a lemma.
 
\begin{lemma}\label{cover}
Suppose $V$ is Riemann surface homeomorphic to a disk  from which a (possibly empty) collection of finitely or countably many pairwise disjoint disks have been removed. Let $\la$ and $\mu$ be two distinct points in $V$. Then there is a  Riemann surface $W,$ homeomorphic to a disk minus a countable collection of pairwise disjoint disks, and an infinite degree holomorphic covering map $h: W \rightarrow V \setminus \{\la, \mu\}$.
\end{lemma}

\begin{proof}
There exists an embedding $e: V\to \widehat{\mathbb{C}}$ such that $e(\lambda)=0$ and $e(\mu)=\infty$.    Consider the exponential map $$Exp(z)=e^z: \mathbb{C}\to \mathbb{C}$$ and set $W=Exp^{-1}(e(V))$. Each component $U$ of $\widehat{\mathbb{C}}\setminus e(V)$ is simply connected and does not contain either $0$ or $\infty$.  Therefore $Exp^{-1}(U)$ is the union of infinitely many simply connected open sets so that $W$ is an open set  with infinitely many holes. Thus  $h= e^{-1}\circ Exp: W\to V$ is  the required map.
\end{proof}

\begin{remark}\label{morecover} We  inductively apply this lemma to  construct a family of  surfaces  and infinite degree covering maps.   The direct limit of this process defines a map that is used in a key step of the proof of the main structure theorem.  

   As in  lemma~\ref{cover}, let $V_0$ be  a topological disk and let $\{U_j \}$ be a (possibly empty) collection of finitely or countably many pairwise disjoint disks in $V_0$.  Set  $U_{0}=V_{0}\setminus \cup_{j\in {\mathbb Z}} U_{j}$ and fix two points, $\la_0$ and $\mu_0$  in $U_0$.  Applying the lemma, we can find a Riemann surface $U_1=V_{1}\setminus \cup_{(j_{1}, j)\in {\mathbb Z}^{2}} U_{j_{1}j}$,  where $V_{1}$ is a topological disk and the $U_{j_{1},j}$ are pairwise disjoint topological disks in $V_{1}$, and an infinite degree holomorphic covering map $h_1: U_1 \rightarrow U_0 \setminus \{\la_0, \mu_0\}$.

    Iterating this process, we choose points $ \la_{n-1}, \mu_{n-1} \in U_{n-1} $ and  obtain Riemann surfaces $U_{n}=V_{n}\setminus \cup_{(j_{n},\cdots  j_{0})\in {\mathbb Z}^{n+1}} U_{j_{n}\cdots j_{0}}$,  where $V_{n}$ is a topological disk and the $U_{j_{n}\cdots j_{0}}$ are pairwise disjoint  topological disks in $V_{n}$,   and  holomorphic covering maps of infinite degree
\[h_{n}: U_{n} \to U_{n-1} \setminus \{ \la_{n-1}, \mu_{n-1} \}.  \]
 \end{remark}

To carry out the proof on the structure of $\cals$, recall the normalized  uniformizing map $\phi_{\la}$ defined in the proof of proposition~\ref{compinv} that conjugates $f_{\la}$ to a linear map near the origin.  We divide the discussion 
  into two parts depending on which of the asymptotic values is on the boundary of  $O_{\la}$,  the domain on which  $\phi_{\la}$ is injective:
    \begin{itemize}
\item Let $\cals_{\la} =\{ \la  \in \cals |  \mu \in \partial{O_{\la}} \}$.
\item Let  $\cals_{\mu} =\{ \la  \in \cals | \la \in  \partial{O_{\la}}  \}$.
\end{itemize}
These sets have a common boundary,
  $\cals_* = \cals_{\la} \cap \cals_{\mu}$, or equivalently,
  \[  \cals_*=\{ \la  \in \cals_\la | \la \in \partial{O_{\la}} \} = \{ \la \in \cals_{\mu} | \mu \in \partial{O_{\la}} \}.\]
 
In section~\ref{fatoucomps} we defined the map $I(\la)$ which is some kind of the inversion in the circle ${\mathcal C}_0$ defined by $|z-\rho/2|=|\rho/2|$.   Using this map we have,

\begin{prop}\label{symsstar} The common boundary set $\cals_*$ is invariant under $I(\la)$. 
  \end{prop}

\begin{proof}  Note that  the affine map  $z\to -z$ conjugates $f_\lambda$  to $f_{I(\lambda)}$. Therefore if $\phi_\lambda(z)$ is the uniformizing map for $f_\lambda,$ then the uniformizing map $\phi_{I(\lambda)}$ for  $f_{I(\lambda)}=f_{\lambda_1,\mu_1}$ is $\phi_{\la}(-z)$. Thus, $$\phi_{I(\lambda)}(\lambda_1)=\phi_{\la}(\mu)\text{ and }\phi_{I(\lambda)}(\mu_1)=\phi_{\la}(\lambda).$$ 

It follows that  $I(\lambda)$ interchanges $\mathcal{S}_\lambda$ and $\mathcal{S}_\mu$ and fixes $\mathcal{S}_*$.   
\end{proof}

Note that the point $\la = \rho$ is in $\cals_*$.     

We saw above, in proposition~\ref{invpts}, that if $\rho$ is real, the  invariant circle  $\calc_0$ of the inversion $I(\la)$ is in $\cals$.   If $\rho$ is real, we can say more.

\begin{prop}\label{cals circle} If $\rho$ is real, then $\cals_*=\calc_0$. 
\end{prop}

\begin{proof}  Let $\sigma(z)=-\overline{z}$.   Then if $\rho$ is real, it is easy to check that for any $z$, $f_\lambda\circ \sigma(z)=\sigma\circ f_\lambda(z)$.   Therefore $$f^n_\lambda(\mu)=f^n_\lambda(-\overline{\lambda})=-\overline{f_\lambda^n(\lambda)}$$  and by proposition~\ref{invpts}, they both converge to $0$.

To show $\cals_*=\calc_0$, we need to show that $|\phi_\lambda(\lambda)|=|\phi_\lambda(\mu)|$ where $\phi_{\la}$ is the uniformizing map defined above such that  $\phi_\lambda(f_\lambda(z))=\rho\phi_\lambda(z)$.  We claim, in fact, that $\phi_\lambda(\mu)=-\overline{\phi_\lambda(\lambda)}$.  

Let $\phi=\sigma \circ\phi_\lambda\circ \sigma (z)$. We claim that $\phi_\lambda=\phi$ since $$\phi(f_\lambda(z))=\sigma \circ \phi_\lambda \circ \sigma (f_\lambda(z))=\sigma \circ \phi_\lambda (f_\lambda(\sigma(z)))=\sigma(\rho \phi_\lambda(\sigma(z)))=\rho \sigma\circ \phi_\lambda\circ \sigma(z)=\rho \phi(z).$$
Then  $\phi_\lambda(\mu)=\phi_\lambda(\sigma(z))=\sigma \phi_\lambda(\lambda))=-\phi_\lambda(-\overline{\lambda})$ as claimed.
\end{proof}

  The following theorem says that the interior $\cals_{\la}^0$ of $\cals_{\la}$ is a topological annulus. It follows that it is connected. 

\begin{thm}\label{calsla} There is a homeomorphism $E: \calslao \rightarrow \AA$ where $\AA$ is a topological annulus.  The inverse map $E^{-1}$
  extends continuously to all points   except one of one of the boundary components of $\AA$.  \end{thm}

\begin{remark} The proof of this theorem is based on a lemma in which we explicitly construct a  homeomorphism  $E$ from $\calslao$ to an annulus.  The construction depends on the choice of a  particular ``model map'' in the period $1$ component $\Omega_1$. 

The proof of the lemma is based on a technique that originally  appeared in the unpublished thesis of Wittner, \cite{Wit}.
The  technique, called ``critical point surgery'',  is used to model pieces of the cubic connected locus on the dynamical plane of a quadratic polynomial.  In \cite{GK},  it was adapted to describe  slices of $Rat_2$, the parameter space of rational maps of degree two with an attracting fixed point.
   Like the $Rat_2$ case, we have two singular values, but unlike that case, our singular values are asymptotic values and our maps are infinite degree and have an essential singularity.   We choose as our model  an   $f_{\la}$ with $\la$ in the period $1$ shell component of  $\calm_{\la}$.  This is the unbounded yellow component in figure~\ref{lambdaplane} and is  denoted  by $\Omega_1$.   As we saw at the end of section~\ref{fatoucomps}, the attracting basin of the fixed point  of $f_{\la}$ is simply connected and completely invariant.   Our model space will be the annulus formed by removing  a dynamically defined disk from this basin.
\end{remark}

 Before we give the proof in detail, we give an outline.    Below, we assume, as we have been doing, that  $\rho$ is fixed and all the functions $f_{\la}$ belong to $\calf_{2}$.

     \begin{enumerate}\label{sketch}
    \item Since the multiplier map is a universal cover of  a shell component to the punctured unit disk, we can find  a $\la_0$ in   $\Omega_1 \subset \calm_{\la}$ such that the multiplier at the fixed point $q_0$  of $f_{\la_0}$ equals  the fixed value $\rho$. This choice is convenient because the map $f_{\la_0}$ is quasiconformally conjugate to a map $\sigma \tan z$ whose  Julia set, by \cite{KK},  is a quasiconformal image of the real line.   In fact, if we take $\rho$ real, $\sigma$ is real, the Julia set of $f_{\la_0}$ is a line parallel to the imaginary axis and the  attracting basin of $q_0$ is a simply connected,  completely invariant  half plane  containing the asymptotic value $\la_0$.  Following the notation in section~\ref{fatoucomps} we denote the basin of $q_0$ by  $K_0$.

 \item     We make the model space by  removing from  $K_0$ a closed dynamically defined topological disk $\Delta$  which contains the fixed point $q_0$ in its interior and $\la_0$ on its boundary. We define the map $E$ from $\cals_{\la}^0$ to  $K_0 \setminus \Delta$   as  follows: to each $\lambda\in \cals_{\la}^0$, we construct a map $\xi_{\la}$ from a subset of the attracting basin $A_\lambda$  of $0$ containing both asymptotic values into the attracting basin $K_0$ of $q_0$ such that $\xi_{\la}(0)=q_0$ and $\xi_{\la}(\mu)=\la_0$;  we set $E(\lambda)=\xi_{\la}(\la)$.  We then  prove that $E$ is injective.

\item   To show $E$ is a homeomorphism, we construct an inverse.  
\begin{itemize}
\item We want to assign a map $f_{\la} \in \cals_{\la}^0$ to each point  $p$ in  $K_0 \setminus \Delta$.  The point $p$ should correspond to the   asymptotic value $\la$ of $f_{\la}$.    Given $p$, we use induction to construct the stable region of a map with two asymptotic values at $\la_0$ and $p$.  At the $n^{th}$ step we obtain a
   domain $U_n$,  homeomorphic to a disk minus an infinite collection of open disks, and a  holomorphic map $Q_n:U_n \rightarrow U_n$ with omitted values $\la_0$ and $p$.  Taking the direct limit  of the pairs   $(U_n, Q_n)$ we obtain a pair $(U_{\infty}, Q_{\infty})$ where $Q_{\infty}:U_{\infty} \rightarrow U_{\infty}$   is a holomorphic covering map with the desired topology; that is, an infinite degree covering map with two asymptotic values.  
       \item We construct a conformal embedding $e:U_{\infty} \rightarrow \CC$ such that $e \circ U_{\infty} = f_{\la} \circ e$ for a unique   $\la \in \cals_{\la}^0$ such that $\xi_{\la}(\la)=p$.  The construction depends on some Teichm\"uller theory.  We give a brief summary of what we need before the construction. 
         \item We extend this inverse map to the points of $\partial\Delta \setminus \{\la_0 \}$ whose image, by construction,  is   $\cals_*$.  Note that the map is not defined for $p=\la_0$; its image  must be a parameter singularity in  $\overline{\cals_*}$.  
    \end{itemize}
    \end{enumerate}

   \medskip
The proof of theorem~\ref{calsla} is contained in the next subsections.

     \subsection{The Model Space}\label{modelspace}
     Every point $\la \in \Omega_1$ corresponds to a function $f_{\la}$ with a non-zero attracting fixed point  denoted by $q_{\la}$; its  attractive basin is denoted by $K_{\la}$.   By  propositions~\ref{compinv} and \ref{Cantordichot}, it is simply connected and completely invariant. In fact,   $f_{\lambda}$ is quasiconformally conjugate to $ t \tan z$ for some real  $t\geq 1$ whose Julia set is the real line (see \cite{DK}), so its Julia set  is the quasiconformal image of a line (see figure \ref{modelbasin} where $\rho=2/3$).  In  figure~\ref{modelbasin}, the cyan colored region is $K_{\la}$ and the yellow region is the basin of $0$, $A_{\la}$.  The black dots are poles on the boundary of $K_{\la}$ and are in the Julia set.   Denote the closure of $K_{\la}$ by $\overline{K_{\la}}$.  It is the analogue of the filled Julia set for a quadratic map.

Because the multiplier map $\nu$ is a universal covering from $\Omega_1$ to $\DD^*$, we can find a sequence of points $\la_j \in \Omega_1$, $j \in \ZZ$ such that $\nu(\la_j)=f_{\la_j}'(q(\la_j))= \rho$.  We choose one, denote it by $\la_0$, and set $q_{0}=q_{\la_0}$. 
We set $Q(z)=f_{\la_0}(z)$ and let $K_0$ denote the attracting basin of $q_0$.   

In figure~\ref{tess} the set $\overline{K}_{0}$ is depicted for   $\rho$  real and $\la_0$ taken as the real solution to $\nu(\la_0)=f_{\la_0}'(q(\la_0))= \rho$.  Since  the multipliers of both attracting fixed points, $0$ and $q_0$, are the same,   
  there is a real $t=t(\rho)$, $\sigma=it$, such that  $Q(z)$ and $it \tan iz =t \tanh z$ are not only quasiconformally conjugate but affine conjugate and the Julia set of $Q(z)$ is a vertical line.  
  
There is a local uniformizing map, which we  denote by $\phi_0$,  $\phi_0: K_0 \rightarrow \CC$,  normalized so that $\phi_0$ maps $q_0$ to $0$, $\phi_0'(q_0)=1$  and $\phi_0$ conjugates $Q$ to $\zeta \to \rho \zeta$ in a neighborhood of $q_0$.  We can extend $\phi_0$ to all of $K_0$ by analytic continuation.  Note that $\phi_0(z)=0$ if and only if $Q^n(z)=q_0$ for some $n$.

\begin{figure}
  \centering
  \includegraphics[width=5in]{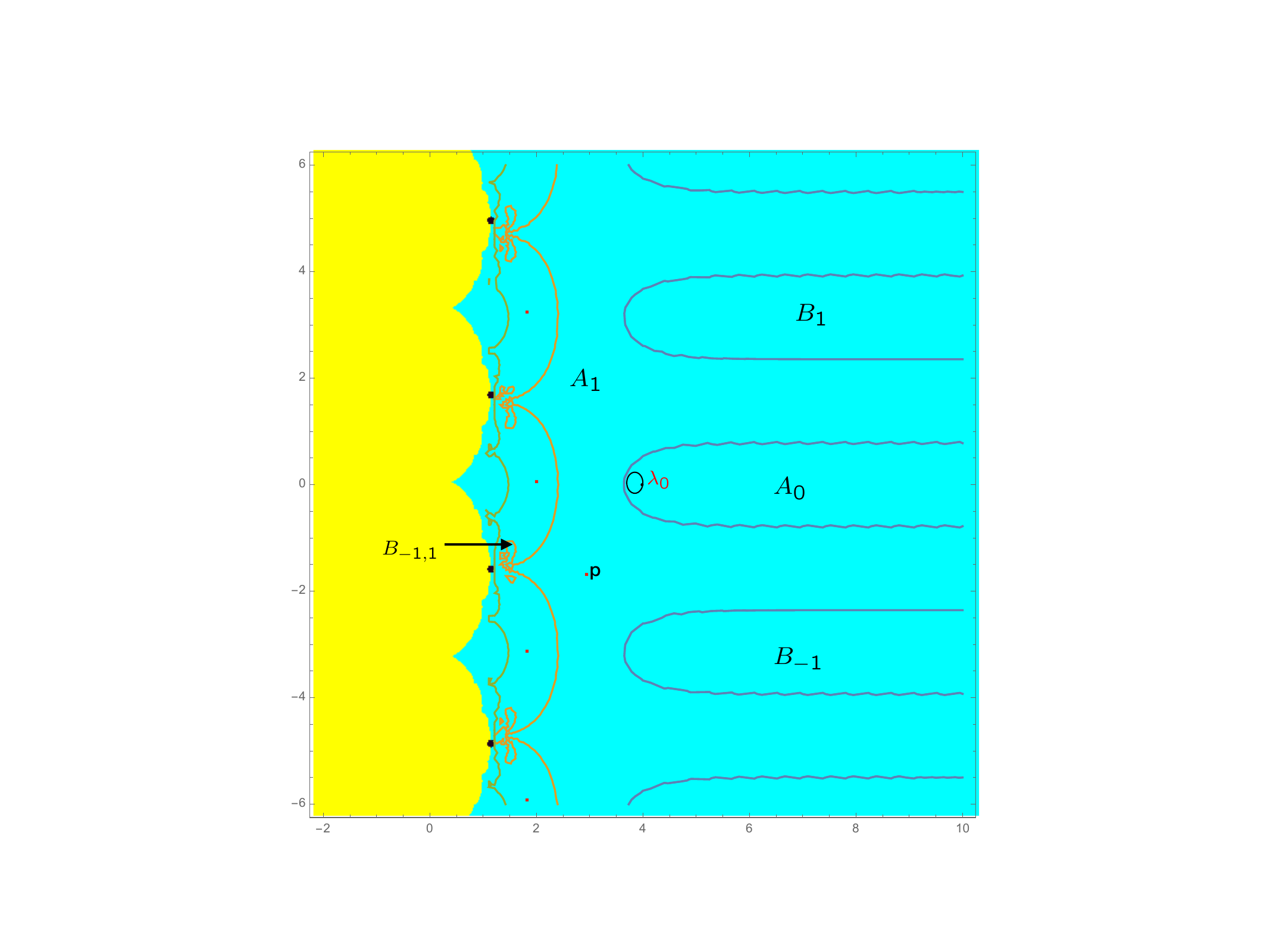}
  \caption{The ``filled Julia set" of $Q(z)$. The black dots are poles.}
  \label{modelbasin}
\end{figure}

Let  $r=|\phi_0(\la_0)|$ and let $\gamma^*=\phi^{-1}_0(re^{i\theta}), \, \theta \in \RR$. It is a simple closed curve;  let $\Delta$ be the closed topological
disk in $K_0$ bounded by $\gamma^*$.  Then $\phi_0$ is injective on $\Delta$ and $\la_0$ is on $\partial\Delta$.

\begin{lemma}\label{the map E} There is an injective holomorphic map $E:\cals_{\la}^0 \rightarrow K_0 \setminus \Delta$.  Set $w=E(\la)$;   $E$ satisfies:
\begin{enumerate}[(i)]
\item For each $\la \in \cals_{\la}$  such that $f_{\la}^n(\la)=0$ for some $n$, $E$ maps it to a preimage of $q_0$; that is,   if  $w=E(\la)$, then $Q^n(w)=q_0$.
\item For each $\la \in \cals_{\la}$ such that $f_{\la}^n(\la)=f_{\la}^m(\mu)$ for some $m,n$, $E$ maps it to a point in the grand orbit of $\la_0$; that is,  if  $w=E(\la)$, then  $Q^n(w)=Q^m(\la_0)$.
\item
As $\la$ tends to the boundary   $\cals_*$ of $\cals_{\la}$, $w=E(\la)$ tends to $\partial\Delta  \setminus \{\la_0 \}$.
\end{enumerate}
\end{lemma}

 \begin{proof}  The map $E$ is defined as follows.  Given $\la \in \cals_{\la}^0$, we defined a conformal homeomorphism $\phi_{\la}$ from the neighborhood $O_{\la}$  in  the attracting basin $A_\la$ to a disk centered at zero with $\phi_{\la}(0)=0$ that conjugates $f_{\la}$ to $ \zeta \mapsto \rho \zeta$.   The map is defined up to affine conjugation and depends holomorphically on $\la$.    We assume here that it is normalized so that $\phi_{\la}(\mu)=\phi_0(\la_0)$.  For each $\la$,  define a  map $$\xi_\lambda=\phi_0^{-1}\circ \phi_\lambda: O_\lambda\to \Delta.$$ From the definitions of $\phi_\lambda$ and $\phi_0$, it follows that $\xi_\lambda(0)=q_0$ and $\xi_\lambda(\mu)=\lambda_0$. Since $f_{\la_0}'(0)=Q'(q_0)=\rho$,  the map $\xi_\lambda$ is a conformal homeomorphism  from $O_{\la}$ to $\Delta$ and it conjugates $f_{\la}$ to $Q$. 
  
  Now we are ready to define the map $E$ from $\mathcal{S}_\lambda^0\to K_0\setminus \Delta$ as
 \[ E(\la) = \xi_{\la}(\la). \]
By construction $E$ satisfies properties~(i) and ~(ii).

 Suppose $\xi_{\la'}(\la')=\xi_{\la}(\la)$. The map $\xi_{\la', \la}=\xi_{\la'}^{-1}\xi_{\la}=\phi_{\la'}^{-1}\phi_{\la}$ restricted to a neighborhood of the origin in the basin $A_{\la}$ defines a holomorphic conjugacy between $f_{\la}$ and $f_{\la'}$ on this neighborhood. Since $\xi_{\la', \la} (\la)=\la'$, we can extend this holomorphic conjugacy by the dynamics of $f_{\la}$ and $f_{\la'}$ to a holomorphic conjugacy, which we still denote by $\xi_{\la', \la}$, defined on the whole stable set $A_{\la}$. Furthermore, by using dynamics of $f_{\la}$ and $f_{\la'}$, we can extend this 
 holomorphic conjugacy to to the Julia set $J_{\la}$ as a topological conjugacy that fixes infinity.  
 Since $\chat=A_{\la}\cup J_{\la}$,  if we denote this extension by $\xi_{\la, \la'}$ again, we have 

$$
\xi_{\la', \la}\circ f_{\la} = f_{\la'}\circ \xi_{\la', \la} \quad \hbox{on $\chat$}.
$$

From the discussion on $J$-stability in subsection 5.1,  we can find a  holomorphic motion $H(z, c): \overline{E}_{\la}\times V\to \chat$ such that $\xi_{\la, \la'}|\overline{E}_{\la} =H_{\la, \la'}=H(\cdot, \la')$ 
is a quasiconformal homeomorphism on $\overline{E}_{\la}=J_{\la}$. 
On $A_{\la}= \chat\setminus \overline{E}_{\la}$, $\xi_{\la, \la'}$ is holomorphic and injective, thus it is conformal. 
Now by a theorem of Rickman  (see~\cite[Theorem 1]{Rick} or~\cite{DH} or~\cite[Theorem 5.1]{J}) it follows that $\xi_{\la, \la'}$ 
is a global quasiconformal mapping of $\chat$.  It follows from the paper of Zheng, (see~\cite[Theorem 3.1]{ZJ}),  that the area of the Julia set $J_{\la}$ is zero and $\xi_{\la, \la'}$ is a global conformal mapping of $\chat$.  
%Thus $\xi_{\la, \la'}$ does not depend on the choice of a curve for $H_{\la, \la'}$.  
Since $\xi_{\la', \la}$ fixes zero and infinity,   $\xi_{\la', \la}(z)=az$.   Equation~\ref{eqn - mult}  implies that $a=1$, $\la=\la'$ which proves that $E$ is injective. 
Note that as we saw in section~\ref{2avs},  there are two choices for $\la'$ but if we require that $\la'$ is the preferred asymptotic value so that   $\la' \in \cals_{\la}$, then  $\xi_{\la',\la}$ is the identity.
   
Property (iii) follows since as  $\la$ tends to the boundary  $\cals_*$ of $\cals_{\la}$,  the asymptotic value $\la$ tends toward the leaf of the dynamically defined level curve containing $\mu$ in the dynamic plane of $f_{\la}$; thus $E(\la)$ tends to a point on  the corresponding level curve, $\partial\Delta \setminus \{ \la_0 \}$ in $K_0$.
\end{proof}

Note that the map $E^{-1}$ is not defined at the point $\la_0$ on $\partial\Delta$ because if it were, the asymptotic values of the function corresponding to the image point would be equal.   Thus the point omitted by $E^{-1}$ would be a parameter singularity, and by continuity, a punctured neighborhood of it would contain points in both $\cals_{\la}$ and $\cals_{\mu}$.    There are only two parameter singularities, $0$ and $\rho/2$;  the latter is a virtual center on the boundary of $\calm_{\mu}$ so  small neighborhoods do not contain points of $\cals_{\la}$.  Therefore the point omitted by $E^{-1}$ is  $0$.   We can extend $E^{-1}$ to $\la_0$ by setting $E^{-1}(\la_0)=0$ so that $E^{-1}(\partial\Delta)$ is the closed curve $\cals_* \cup \{0\}$. 

 \begin{remark}  The map $\xi_{\la}$ ties together   the attractive basin of the origin in the dynamical space  of $f_{\la}$ and the attractive basin of $q_0$ in the dynamical space of  $Q$ with the parameter space of $f_{\la}$. \end{remark}

\subsection{Construction of an inverse for $E$.}

\subsubsection{Dynamic decomposition of  $K_0$.}
 To define inverse branches $R_j$ of $Q$ on the $K_0$, let $l^*$ be the gradient curve joining $Q(\la_0)$ to $\la_0$ in $\Delta$ and let  $l \in Q^{-1}(l^*)$ be the curve joining $\la_0$  to infinity.  Remove the line $l$ from $K_0$ and define an inverse branch on its complement by the condition $R_0(q_0)=q_0$.      Label the other branches as
 $R_j(q_0)=q_0+\pi i j=q_j$.   This is equivalent to choosing a principal branch for the logarithm.  Having made this choice, we can extend the $R_j$ analytically to all of $K_0$.   
  Denote the preimages of $q_0$ under $Q^{-1}$ by $q_j$, enumerated so that $q_0$ is fixed, and denote the inverse branch of $Q$ that sends $q_0$ to $q_j$ by $R_j$.    Denote the upper and lower sides of the line $l$ by $l^+$ and $l^-$ and let $l_j=R_j(l^-)$, $l_{j+1}=R_j(l^+)=R_{j+1}(l^-)$.  Then $R_0$ is a homeomorphism between the open region bounded by the lines $l_0$, $l_1$ and $l$ onto $K_0 \setminus l$ and  $R_j$, $j \neq 0$ is a homeomorphism from the open region between $l_j$ and $l_j+1$ onto $K_0 \setminus l$.

  If $\rho$ and $\la_0$ are real, this choice for the logarithm agrees with the labeling of the poles  and inverse branches in  section~\ref{combinatorics}  where  
 $R_0=g_{\la_0,0}$,    the branch of $Q^{-1}=f_{\la_0}^{-1}$ that fixes the origin.   This is the labeling  in figure~\ref{tess}.   If $\rho$ and/or $\la_0$ isn't real, and a different branch of the logarithm is chosen,  there could be a shift by some $k$ in the labelling.  It would be the same shift throughout the rest of the paper so would not change the essence of the argument.    
 
  \begin{figure}
  \centering
  \includegraphics[width=5in]{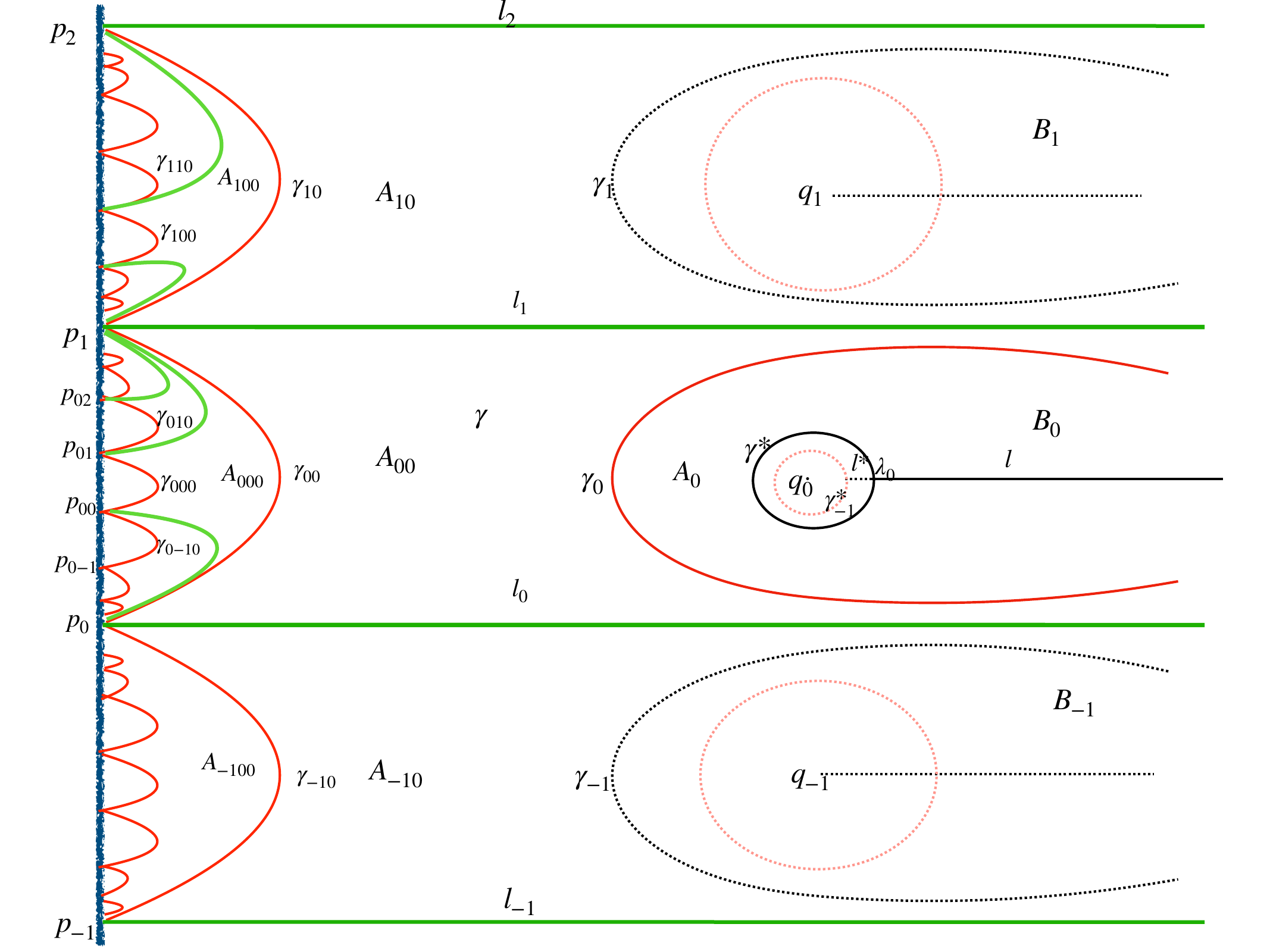}
  \caption{Domains and curves in the construction.}
  \label{tess}
\end{figure}

Recall that  $\gamma^{*}$  is the boundary of $\Delta$ in $K_0$ and it contains $\la_0$.   Then, because $\la_0$ is an omitted value of $Q$,  the curves $\{\gamma_{j} \}=R_j(\gamma^*), \, j \in \ZZ$, are a countable collection of bi-infinite disjoint curves whose infinite ends  approach infinity asymptotic to the lines $l_{j}$ and $l_{j+1}$. 
Thus $R_j(\Delta)$ is  an unbounded domain, with boundary $\gamma_j$, that contains $q_j$. Note that $R_0(\Delta)$ contains the removed line $l$.  We label the complementary components of the  $\gamma_j$ as follows (see figure~\ref{tess}): 

\begin{itemize}
\item  $A_0=R_0(\Delta)$ is the component of the complement    of  $\gamma_0$ containing the fixed point $q_0$ and the point $\la_0$. 

\item $B_{j}, \, j \in \ZZ \setminus \{0 \}$ are the components of the complement of $\gamma_j$ containing the non-fixed preimages $q_{j}$ of $q_0$ and

\item $C_0$ is the common  complementary component in $K_0$ of $A_0$ and  all the $B_{j}$.
\end{itemize}

To define the second preimages of $q_0$ and $\gamma^*$ we need two indices.  Thus $q_{j_2 j_1}=R_{j_2}R_{j_1}(q_0)$ and $\gamma_{j_2 j_1} =R_{j_2}R_{j_1}(\gamma^*)$  where $j_1, j_2, \in \ZZ$.  They divide $K_0$ into domains as follows  (see figure~\ref{tess}):

\begin{itemize}
\item   Since $A_0$ is simply connected and contains one asymptotic value, $Q: A_1=Q^{-1}(A_0)\to A_0\setminus \{\lambda_0\}$ is a universal covering.    Set $A_{j0}=R_j(A_0)$; it is bounded by $l_j$,  $l_{j+1}$ and $\gamma_{j0}$.   
 Each $\{\gamma_{j0}\}$ joins the   pole $R_j(\infty)$ to the pole $R_{j+1}(\infty)$; these two poles are different but adjacent because the infinite ends of $\gamma_0$ are on opposite sides of the line $l$ defining the principal branch.   

\item  Since $B_{j_1}$ is simply connected and  contains no asymptotic value for any $j_1\neq 0$, each component of $Q^{-1}(\gamma_{j_1})$ is homeomorphic to $\gamma_{j_1}$. The curves $\gamma_{j_2 j_1}$ bound domains containing the preimages $q_{j_2 j_1}$.  Label these domains $B_{j_2 j_1}=R_{j_2}(B_{j_1})$.

\item There are domains $C_{j0}=R_j(C_0)$.
\end{itemize}

Inductively we have curves 
\[ \gamma_{j_n \ldots j_1} = R_{j_n}(\gamma_{j_{n-1}} \ldots j_1). \]
and the regions they define as follows  (see figure~\ref{tess});\\
 \begin{itemize}
 \item    $A_{n}=R_0(A_{n-1})$;  it contains $q_0$ and $\lambda_0$. It also contains all preimages of $q_0$ up to order $n-1$ but not  those of order $n$.  It is bounded by a curve $Q^{-n}(\gamma_0)$ that is a union of open arcs with endpoints at adjacent prepoles of order $n$.  These are the red curves without labels closest to the vertical line in figure~\ref{tess}.  
 \item $B_{j_n j_{n-1} \ldots  j_1} =R_{j_n}(B_{j_{n-1} \ldots j_1})$; it contains the preimage $q_{j_n j_{n-1} \ldots j_1}$ of $q_0$.  These are not shown in the figure.  They are bounded by a single curve with a boundary point at a prepole of order $n-1$. 
  \item $C_{j_n j_{n-1} \ldots, 0} =  R_{j_n}(C_{ j_{n-1} \ldots  0} ) $.
  \end{itemize}

\subsubsection{Inductive construction of the  pair $(U_{\infty}, Q_{\infty})$}
 See figure~\ref{setU}.
\begin{itemize}
\item Pick $p \in K_0 \setminus \Delta$.  In figure~\ref{setU}, $p$ is in $A_{00}$. Following the outline above, part (3), we construct a map with the asymptotic values $p$ and $\lambda_0$ as follows.    Let $\hat{p}_j=R_j(p)$;  in the figure the $\hat{p}_j$ are in $A_{j00}$.  
    Let $N$ be the smallest integer such that  $p$ is in $A_{N} \cup   B_{j_N, j_{N-1} \ldots, j_1}$.  The boundaries of the sets in this union are the level sets  $\phi_{0}^{-1}(\rho^{-N} re^{i\theta})$ where, as above,  $r=|\phi_0(\la_0)|$. 
 For every small  $\epsilon>0$, one component of the level set   $\phi_0^{-1}( (\rho^{-N} r +\epsilon)e^{i\theta})$ is an analytic curve, except at the prepoles of order $N-1$. It  bounds a simply connected domain containing $A_N$;   it thus contains the points $p, q_0, \la_0$ and the curves $\gamma_{j_N, \ldots j_1}$,  but none of preimages $\hat{p}_j$ of $p$.  Fix $\epsilon$, and denote the resulting  domain  by $U$. Its boundary is  denoted by the dotted black curve in figure~\ref{setU}.  Since it is contained in the attracting basin of $q_0$, $Q(U) \subset U$.   Moreover, since $U$ does not contain any of the points  $\hat{p}_j$, $p \not\in Q(U)$.
 Set  $\widetilde{U}= U \setminus \{\la_0 , p\}$.

 \begin{figure}
  \centering
  \includegraphics[width=5in]{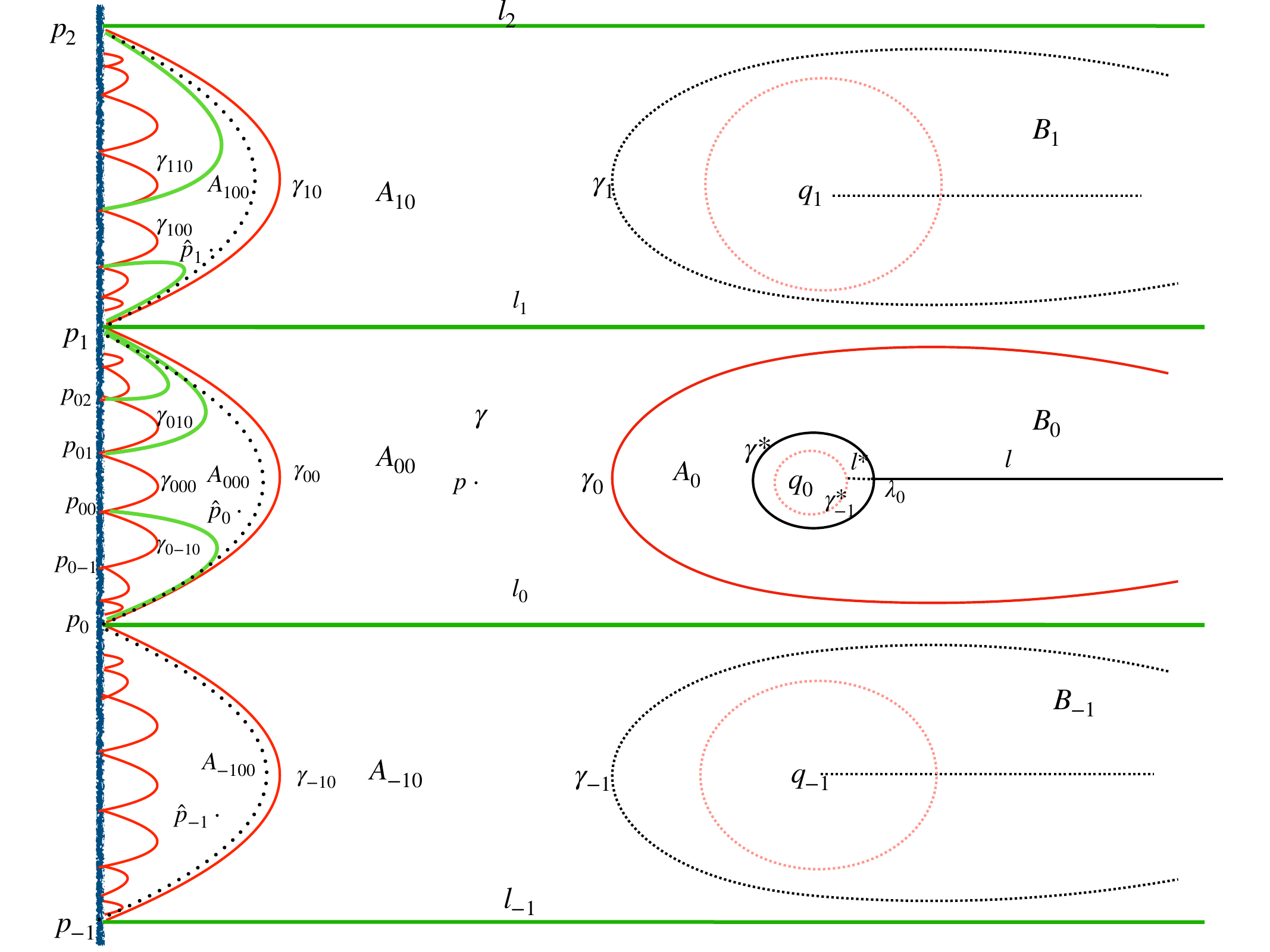}
  \caption{The point $p$ is in $A_{00}$ and the set $U$ is the region to the right of the dotted curves. }
  \label{setU}
\end{figure}

\item Lemma~\ref{cover} implies there is a holomorphic unramified covering map $\Pi_1: U_1 \rightarrow \widetilde{U}$ where  $U_1$ is a Riemann surface  that is topologically a disk minus a countable set of topological disks.

\item  Note that $Q : U \rightarrow Q(U)$ is a holomorphic universal covering map with omitted value $\lambda_0$. Set $U_1'=\Pi_1^{-1}(Q(U))$; since $\lambda_0\in Q(U)$ and $p\notin Q(U)$, it is a topological disk so that $\Pi_1 : U_1' \rightarrow Q(U)$ is also a holomorphic unramified covering map that omits the value $\lambda_0$.

    \centerline{
\xymatrix{U_1\ar[d]_{\Pi_1}\ar@{.>}[r]^{Q_1=i_1\circ \Pi_1} & U_1'\ar[d]^{\Pi_1}\ar@{^{(}->}[r]& U_1\\
U \ar[r]^Q\ar@{.>}[ru]^{i_1}& Q(U)\ar@{^{(}->}[r]& U
}
}
\item Since both $\Pi_1$ and $Q$ are regular coverings whose domains are simply connected, we can lift to  obtain  a conformal map $i_1: U\to U_1'$  such that $\Pi_1 \circ i_1=Q.$     The choice of inverse branch will affect $i_1$ but the argument works for any choice. We now define  $Q_1: U_1 \rightarrow U_1'$ by $Q_1= i_1\circ \Pi_1$.  It is an unramified regular  infinite to one holomorphic endomorphism and omits the values $i_1(\lambda_0)$ and $i_1(p)$. Moreover, $$ Q_1\circ i_1=i_1\circ \Pi_1\circ i_1=i_1\circ Q;$$ that is, $i_1$ conjugates $Q$ and $Q_1$, therefore $Q_1$ fixes $i_1(q_0)$. 

We may, without loss of generality, assume $i_1(\lambda_0)=\lambda_0$, $i_1(p)=p$ and $i_1(q_0)=q_0$. %

\item
Now set $Q_{0}=Q$, $U_{0}=U$ and  $U_{0}'=Q(U)$.   We proceed by induction:  we assume that for $1 \leq j \leq n-1$ we have

\begin{enumerate}
\item Domains $U_j$, homeomorphic to an open disk   from which infinitely many open disks been removed, and infinite to one unramified covering maps $\Pi_j:  U_j \rightarrow U_{j-1}$ with two asymptotic values.
\item Holomorphic endomorphisms,  $Q_j: U_j \rightarrow U'_j\subset U_j$, that are infinite to one,  unramified, have one fixed point and two asymptotic values. 
\item  Conformal maps $i_{j}: U_{j-1} \to U'_{j}$ satisfying
$$ Q_j\circ i_j=i_j\circ Q_{j-1}.$$
\end{enumerate}

For the inductive step, we use Remark~\ref{morecover} to obtain the holomorphic unramified covering map $\Pi_n: U_{n} \rightarrow U_{n-1}$ where $U_{n}$ is homeomorphic to $U \setminus \{U_{j_{n-1} \ldots j_1 j},  \, \, ( j_{n-1}, \ldots, j_1,j)  \in \ZZ^{n} \}$.   As in the first step we set  $U_{n}'=\Pi_{n}^{-1}(Q_{n-1}(U_{n-1}))$.    Both
\[  Q_{n-1}: U_{n}' \rightarrow Q_{n-1}(U_{n-1}) \mbox{ and }   \Pi_{n}:U_n \rightarrow Q_{n-1}(U_{n-1}) \]
are unramified coverings with asymptotic values $\lambda_0$ and $p$.   Lemma~\ref{cover} implies there are infinitely many choices for a holomorphic isomorphism
\[ i_{n}: U_{n-1} \rightarrow U_{n}' \mbox{  satisfying  } \Pi_{n} \circ i_{n}= Q_{n-1} \mbox{  on   } U_{n-1}.  \]
Making one such choice (the choice doesn't matter) we define $Q_{n} : U_{n} \rightarrow U_{n}$ by
$  Q_{n}=i_{n} \circ \Pi_{n}$ so that 
\[ Q_n\circ i_n= i_n \circ \Pi_n\circ i_n=i_n\circ Q_{n-1}. \]
Therefore $i_{n}$ conjugates $Q_n$ to $Q_{n-1}$ and the induction hypotheses are satisfied, completing the inductive step.

\item
The direct limit $U_{\infty}$ of the system $(U_n, i_n)$ is the quotient
\[ \cup_n U_n / \sim  \]
where the equivalence relation  is defined by the  identifications,  $ z \sim i_n(z)$, and the equivalence class is denoted by $[z]$.  The Riemann surface $U_{\infty}$ has infinite type.
There is an infinite unramified holomorphic covering map $Q_{\infty}$ defined by
\[  Q_{\infty}([z])=[Q_n(z)],  \, \, z \in U_n  \]
that has two omitted values  $[\lambda_0]$ and $[p]$.  It also fixes  $[q_0]$ and since the maps $\Pi_n$ and $i_n$ are holomorphic we have   $Q_{\infty}'([q_0])=\rho$.     

 Topologically $U_{\infty}$ is the complement in $\CC$ of a Cantor set $C$ isomorphic to the space of  infinite sequences  $\Sigma_{\infty} = s_1, \ldots s_{n-1},   s_n  \ldots,  s_j \in \ZZ$ together with the finite sequences $\Sigma_{n+1}=s_1, \ldots, s_n, \infty $ of length $n+1$.  The map $Q_{\infty}$ is conjugate to the shift map on $C$.   See \cite{Mo}.
     \end{itemize}

  The final step of the proof is to show  there is a  conformal embedding $e: \, U_{\infty} \rightarrow \CC$ such that
  \[ e \circ Q_{\infty} = f_{\la} \circ e \] for some $\la \in \cals_{\la}^0$ with $\xi_{\la}(\la)=\la_0$.    To do this, we first  give a brief informal review the results we need from Teichm\"uller theory  and the   theory of mapping classes of tori and punctured tori.  We refer the reader to \cite{Bir} for a full discussion and \cite{GK} for a discussion analogous to what we need here.
  
\subsubsection{Teichm\"uller theory}\label{teichth}
 Fix $\la \in \cals_{\la}^0$ and set $f=f_{\la}$. 
 
    \begin{definition}
 Let $QC(f )$ be the set of  quasiconformal maps  $h :\CC \rightarrow \CC$   such that  $g=h \circ f \circ  h^{-1}$ is  meromorphic.  Since $g$ is a meromorphic infinite to one unbranched cover of the plane with two omitted values, by the corollary to Nevanlinna's theorem, corollary~\ref{Nevcor} ,  it   is affine conjugate to a map $f_{\la'} \in \calf_2$ and we choose the conjugacy so that $\la'$ is the preferred asymptotic value.

We define the Teichm\"uller equivalence relation on $QC(f)$ as follows:  elements $h_1, h_2$ of $QC(f )$ are equivalent if there is an affine map $a$ and an isotopy from $h_1$ to $a\circ h_2$ through elements of $QC(f )$.   The quotient space of $QC(f )$ by this equivalence relation is called the Teichm\"uller space $Teich(f)$ with basepoint $f$.
  \end{definition}

   Let $QC_0(f)$ denote the elements of $QC(f)$ that conjugate $f$ to itself and $QC_0^*(f)$ those that preserve the marking of the asymptotic values.
 \begin{definition} The {\em  mapping class group}, $MCG(f)$,   is the quotient of $QC_0(f)$ by the Teichm\"uller equivalence relation and the  {\em pure  mapping class group}, $MCG^*(f)$,   is the quotient of $QC_0^*(f)$ by the Teichm\"uller equivalence relation. The {\em moduli space} and {\em pure moduli space}  are defined as the quotients  $\rmm(f) =Teich(f)/MCG(f)$  and $\rmm^*(f) =Teich(f)/MCG^*(f)$.
 \end{definition}

 \begin{remark}\label{slice} Because we are working in a  dynamically natural slice of $\calf_2$ defined by the conditions that $0$ is fixed and has multiplier a fixed $\rho$, we restrict our considerations here to the slice $Teich(f, \rho) \subset Teich(f)$ of equivalence classes of quasiconfomal maps $h$ such that $h \circ f \circ h^{-1}$ has a fixed point with multiplier $\rho$.   The mapping class group and pure mapping class group act on $Teich(f,\rho)$.  The $\la$ parameter plane is identified with the pure moduli space  $\rmm^*(f,\rho)$.
  For readability below, since we always assume we are in this slice, we  drop the $\rho$ from the notation. 
 \end{remark}

 Because the quasiconformal maps conjugate the dynamics, and the dynamics are controlled by the orbits of the asymptotic values, the space $Teich(f_{\la})$ is related to the Teichm\"uller space of a twice punctured torus defined by the dynamics.  We explain this here. 

 \begin{definition} The points $z_1,z_2$ are {\em grand orbit equivalent} if there are integers $m,n \geq 0$ such that $f_{\la}^m(z_1)=f_{\la}^n(z_2)$.  They are {\em small orbit equivalent} if for some $n>0$, $f_{\la}^n(z_1)=f_{\la}^n(z_2)$.  Denote the grand orbit equivalence classes by $[z]$. 
 \end{definition}

 Now $\phi_{\la}(z)=0$ if and only if $z$ is grand orbit equivalent to $0$.   Let $\hat{A_{\la}}$ denote  the complement of the grand orbit of $0$ in $A_{\la}$.    We have

 \begin{lemma}\label{orbitequ}  The restriction of $\phi_{\la}$ to $\hat{A_{\la}}$ is a well defined map from each small equivalence class to a   point in $\CC^*$.
 \end{lemma}

 \begin{proof} If $z_1,z_2$ are small orbit equivalent there is some integer $N$ such that for all $n \geq N$, $f_{\la}^n(z_1)=f_{\la}^n(z_2)$.  Moreover, for  all large $n$, $f_{\la}^n(z) \in O_{\la}$ and, since  $\phi_{\la}$ is injective on $O_{\la}$, the lemma follows.
 \end{proof}

 Let $\Gamma_{\rho}$ be the group generated by $z \mapsto \rho z$ in $\CC^*$.  The projection $\tau_{\rho}: \CC^* \rightarrow \CC^*/\Gamma_{\rho} = T_{\rho}$  is a holomorphic covering map onto a torus $T_{\rho}$.  Following common usage, we say that its modulus is $\rho$.  Set $T=T_{\rho}$ since $\rho$ is fixed in this discussion.   
 
  Define the composition of $\phi_{\la}$ with $\tau_{\rho}$ by
 \[  \Phi_{\la}:\hat{A_{\la}} \stackrel{\phi_{\la}}\rightarrow \CC^* \stackrel{\tau_{\rho}}\rightarrow T. \]

By lemma~\ref{orbitequ}, we see that $\Phi_{\la}$ identifies  $\hat{A_{\la}}$ in the dynamical space of $f_{\la}$ with the torus $T$ because each grand orbit  in $\hat{A_{\la}}$ maps to a unique point on $T$.   Notice that $T$ depends only on $\rho$ and not on $\la$.   Let $\gamma^*$ be the level curve through the asymptotic value $\mu$ in $\hat{A_{\la}}$, and $\beta$ its projection on  $T$.  

Since $\la \in \cals_{\la}^0$, the orbit of $\mu$ accumulates on $0$ so it cannot be in the grand orbit of $0$.  It is possible that $\la$ is in the grand orbit of zero, or that for some $m,n$, $f_{\la}^n(\la)=f_{\la}^m(\mu)$.  This can happen only on a discrete set and we assume here that it does not happen for the $\la$ we chose.  

There are two special points on $T$, the points $\la^*=\Phi_{\la}(\la)$ and $\mu^*=\Phi_{\la}(\mu)$.     We mark them so that $\la^*$ is the preferred point.   Let 
$T^2_{\la}= T \setminus \{ \la^*, \mu^* \}$.     
Let $A_{\la}^*=\Phi_{\la}^{-1}(T^2_{\la})$;   then $A_{\la}^* \subset A_{\la}$ is the complement of the grand orbits of $0$ and the asymptotic values.  It is easy to see   that $\Phi_{\la}:A_{\la}^* \rightarrow T^2_{\la}$ is a covering projection.   
 
The Teichm\"uller space $Teich(T^2_{\la})$ is defined as the set of equivalence classes of quasiconformal maps, $[H]$, defined on $T^2_{\la}$, where, as above, the equivalence is through isotopy.    The pure mapping class group $MCG_*(T^2)$ and   pure moduli space 
$\rmm^*(T^2)$ based at $T^2_{\la}$ are defined as for $Teich(f)$: the pure mapping class group consists of equivalence classes $[H]$ that map  $T^2_{\la}$ to itself preserving the marking  and the pure moduli space is formed by identifying   points congruent under the pure mapping class group. Thus the map $\Phi_{\la}$ induces a  map $\Psi: Teich(f) \rightarrow Teich(T^2)$. By standard arguments, see e.g.  \cite{McMSul}, $\Psi$ is a covering map so there is an injection on fundamental groups which translates to an injection of pure mapping class groups:
\[  \Psi_*:  MCG_*(f) \rightarrow MCG_*(T^2).   \]

 Since a quasiconformal map  $H \in Teich(T^2_{\la})$ is not necessarily  the projection by $\Phi_{\la}$ of an $h$ defined on $A_{\la}^*$,  we need to characterize those that are.  
 To do this, we need to understand the image $\Psi_*(MCG_*(f)) \subset MCG_*(T^2)$.   

First of all, to remain in the slice, we require that $\omega(H(T^2_{\la}))$, the torus obtained by applying the ``forgetful map'' $\omega$ that fills in the punctures, is conformally equivalent to $T$ and preserves the isotopy class of $\beta$.   

Suppose $\tilde{\alpha}$ is a curve in $A_{\la}^*$ with initial point $\mu$ and endpoint $\la$ and $[h] \in MCG_*(f)$.   Then $h(\tilde{\alpha})$ has the same property.   The map $H=\Phi_{\la} \circ h \circ \Phi_{\la}^{-1}$ determines a point in $MCG_*(T^2)$ that maps the curve $\alpha^*$  on $T^2_{\la}$ joining $\mu^*$ to $\la^*$ to a curve $H(\alpha^*)$ with the same endpoints.  
   
   Every curve $\alpha'$ on $T^2$ that joins $\mu^*$ to $\la^*$ has  lifts $\Phi_{\la}^{-1}(\alpha')$ whose initial point is at a preimage of $\mu^*$;  let $\tilde{\alpha}'$ be the lift at the asymptotic value $\mu$.   The endpoint of $\tilde{\alpha}$ is in the grand orbit of $\la$, but it isn't necessarily at $\la$.   
Therefore, in order to construct maps in $Teich(f)$ from maps in $Teich(T^2)$, which we do below,  we need to know that we can find those curves $\alpha$ whose lift to $\mu$ lands at $\la$.  Let $\alpha^*$ be such a curve on $T^2$.  

That we can always find these curves is proved in 
 \cite{Bir} where there is a   full treatment of mapping class groups of surfaces.  For  a more detailed discussion analogous to the situation here see \cite{GK}. 

    \begin{figure}
     \centering
  \includegraphics[width=5in]{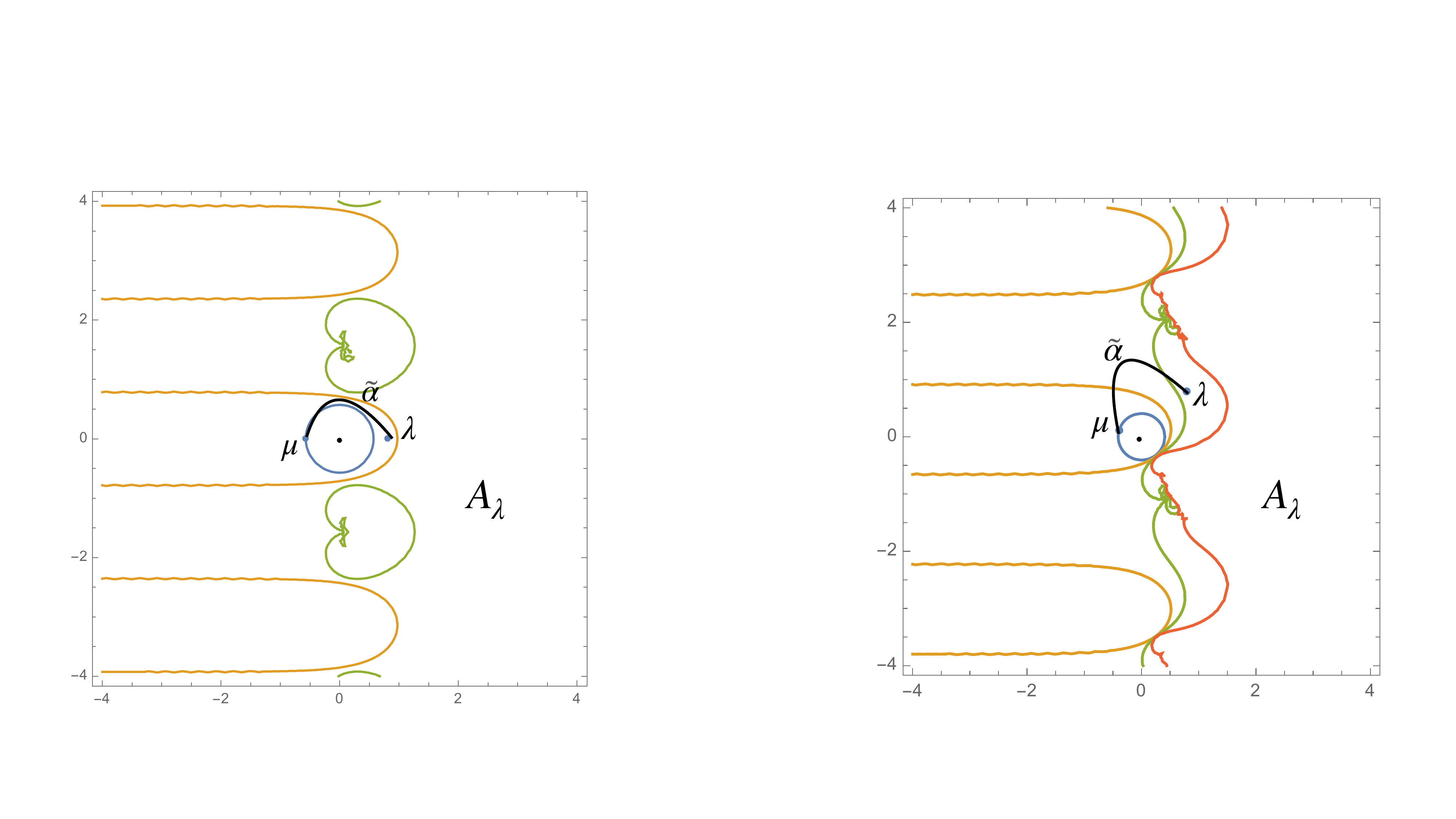}
  \caption{Two examples where the lifted curve is the one we need. }
  \label{lifts}
\end{figure}

In figure~\ref{lifts} we show how the region $A_{\la}$ is divided into fundamental domains that project to $T$ for two different values of $\la$.   In both, $\mu$ is on $\gamma^*$, the boundary of $O_{\la}$, drawn in blue.  The orange curves are the first pullbacks of $\gamma^*$ by $f_{\la}$.  The domain bounded by $\gamma^*$ and one of the orange curves defines a fundamental domain for $\Gamma_{\rho}$. In the left figure, $\la$ is in that fundamental domain.  The green curves are the next pullback, and on the right figure, the red curve is the third pullback and  $\la$ is in a fundamental domain between the second and third pullbacks.  

 \subsubsection{Construction of the embedding $e$}
 \label{map e}
 
  We now construct the   conformal embedding $e: \, U_{\infty} \rightarrow \CC$ such that
  \[ e \circ Q_{\infty} = f_{\la} \circ e \] for some $\la \in \cals_{\la}^0$ with $\xi_{\la}(\mu)=\la_0$.  

Delete the grand orbits of  $[q_0], [\lambda_0]$ and $[p]$ from $U_{\infty}$ to obtain a domain $U_{\infty}^*$.  As we did above for $A_{\la}$, we  form the projection by the grand orbit equivalence
 \[ \Phi_{\infty}: U_{\infty}^* \rightarrow T_{\infty}^2 = T \setminus \{  \Phi_{\infty}(p), \Phi_{\infty}(\la_0) \}\] where again,  $T$ is a torus of modulus $\rho$. 
 
As above, there is some  $\alpha_{\infty}$ that is  a curve on $T^2_{\infty}$ with initial point  $ \Phi_{\infty}(\la_0)$ and endpoint $\Phi_{\infty}(p)$ whose lift to $Q_{\infty}$ at $\la_0$ is a curve $\tilde{\alpha}_\infty$ joining $\la_0$ to $p$.
  
 Let $H: T_{\la}^2 \rightarrow T_{\infty}^2$ be an orientation preserving homeomorphism that preserves the labeling of the punctures and satisfies $H(\alpha_*)=\alpha_{\infty}$.   Then it lifts to a topological conjugacy $h$ between $f_{\la}|A_{\la} $ and $Q_{\infty}$.    

 \[ \xymatrix@1{
 U_{\infty} \ar[d]^{Q_{\infty}} & A_{\la} \ar[d]^{f_{\la}}\ar[l]_h  \ar[r]^g & A_{\la(p)}  \ar[d]^{f_{\la(p)}}\\
  U_{\infty} \ar[d]^{\Phi_{\infty}} & A_{\la} \ar[d]^{\Phi_*}  \ar[l]_h\ar[r]^g & A_{\la(p)} \\
  {T_{\infty}^2} & {T_{\la}^2} \ar[l]_H}\]

  We may assume that $H$ is quasiconformal with Beltrami differential $\nu_{T_{\la}^2}$.  
  and use $\Phi_*$ to lift to a Beltrami differential $\nu$ on $A_{\la} $  compatible with the dynamics.  We set $\nu=0$ on the complement of $A_{\la} $ (the Julia set of  $f_{\la}$), and note that because the map is hyperbolic, this set has measure zero.  We now invoke the measurable Riemann mapping theorem, \cite{AB},  to obtain a quasiconformal homeomorphism $g: \hat\CC \rightarrow \hat\CC$ fixing $0$ and $\infty$, and so  unique up to scale,  such that $g \circ f_{\la} \circ g^{-1}$ is holomorphic.    By Nevanlinna's theorem, theorem~\ref{Nev},  we can assume $g$  is normalized so that $g \circ f_{\la} \circ g^{-1}$ is of the form $f_{\la(p),\rho(p)}$ for some $\la(p )$ where   $\la(p)=g(\la)$  and $\mu(p)=g(\mu)$ is on the boundary of $O_{\la(p)}$,
  the region of injectivity of the uniformizing map at the origin.  Since $g$ is compatible with the dynamics, and both tori $T_{\la}^2$ and $T_{\infty}^2$ have modulus $\rho$; it follows that $g'(0)=\rho$ also.   Thus $\la(p) \in \cals_{\la}^0$ and
the map $e=g \circ h^{-1}$ is the required embedding.

\medskip
To complete the proof we need to show that the correspondence  $p \to \la(p)$ is an inverse of the map $E$.

\[\xymatrix{
K_0 \\
   U_0 \ar@{^{(}->}[u]  \ar[r]^{i_{\infty}} & U_{\infty} \ar[r]^e & {A_{\la(p)} \ar[llu]_{\xi_{\la(p)}}\subset \CC}} \]

 By our construction, $i_{\infty}$ is the direct limit of the maps $i_n$. It  satisfies
\[  e \circ i_{\infty}(p) = \la(p). \] The second asymptotic value of $f_{\la(p)}$ is $\mu(p)$.   By definition, $\xi_{\la}(\mu(\la(p)))=\la_0  \mbox{ and } \xi_{\la}(\la(p))=p \in U_0 \subset K_0$.

In the non-generic cases, the point $p$ in  $K_0 \setminus \Delta$ is either in the grand orbit of the fixed point $q_0$ or the other asymptotic value $\la_0$ and the quotient of $U_{\infty}$ by the grand orbit relation is a once punctured torus.   The construction of the inverse of $E$ is analogous, but simpler in these cases and again yields a unique $f_{\la} \in \cals_{\la}^0$.

\medskip

If we choose $p$ on $\partial\Delta$, the function $f_{\la(p)}$ will have both its asymptotic values on the boundary of   $O_{\la(p)}$.  Only one choice, however, preserves the marking.   

By the Measurable Riemann Mapping Theorem, the quasiconformal map $g$ depends holomorphically on the parameter $p$. Thus, as we vary $p$ analytically along  $\partial\Delta \setminus \{\la_0 \}$, the image $e(p)$ defines an analytic curve  $\cals_*$ in $\cals$.   The construction fails if $p=\lambda_0$ because as $p$ approaches $\la_0$, the limit point   on the analytic curve in $\cals$ is a parameter singularity;   in the construction of $f_{\la}$ from the model,  as $p \to \la_0$, $\la \to 0$.  Therefore we can extend $E^{-1}$ by continuity so that $E(0)=\la_0$; therefore $\cals_* \cup \{0\}$ is homeomorphic to a circle.  

Since the model $K_0 \setminus \Delta$  is topologically an annulus $\AA$,  the above paragraph shows that $E$ extends as a map from the boundary component $\cals_*$ of $\cals_{\la}$ to a boundary component  of $\AA$.

\subsection{Topology of the shift locus}
We are now ready to complete the proof of the Main Structure Theorem.

\begin{thm}[Topology of the shift Locus]\label{thm:shift locus} $\cals$ is homeomorphic to a punctured annulus;  that is, there is a homeomorphism $\Phi: \cals \rightarrow \hat{\CC} \setminus \{0,1,\infty\}$.     \end{thm}

\begin{proof}   We begin by  recalling the relation between $\calm_{\la}$ and $\calm_{\mu}$ given by the inversion $I(\la)= -\mu$  defined in section~\ref{fatoucomps}  that shows 
\[ f_{\la}(z)= f_{-\mu}(-z). \]
It follows that if $\la \in \calm_{\la}$ and $f_{\la}^m(\la)$ tends to a periodic orbit $\mathbf z = \{z_0, z_1, \ldots, z_n \}$ then												
$f_{-\mu}^m(-\mu)$ tends to the orbit $-\mathbf z = \{-z_0, -z_1, \ldots, -z_n \}$ and $-\mu \in \calm_{\mu}$.    This proves

\begin{prop}\label{inversion} The inversion $I: \calm_{\la} \rightarrow \calm_{\mu}$ defined by
\[ I(\la)=-\mu = \frac{\la}{2\la/\rho -1}  \]  maps shell components of  period $n$ in $\calm_{\la}$ to shell components of period $n$ in $\calm_{\mu}$.
\end{prop}

This is illustrated in figure~\ref{lambda and rho}.  The large green region is the shift locus,  $\calm_{\la}$ is the complementary region on the right and $\calm_{\mu}$ is the complementary region to the left, surrounded by the shift locus.  
The circle of inversion is drawn in   figure~\ref{lambda and rho} where $\rho=2/3$.  In this figure, since $\rho$ is real, by proposition~\ref{cals circle}, it is $\cals_*$.  For arbitrary fixed $\rho$, $\cals$ is the image of $\partial \Delta \setminus \{\la_{0}\}$ under $E^{-1}$; thus  $\cals_*\cup \{0\}$, which we still denote as $\cals_*$, is a topological circle. 

\begin{figure}
  \includegraphics[width=2in]{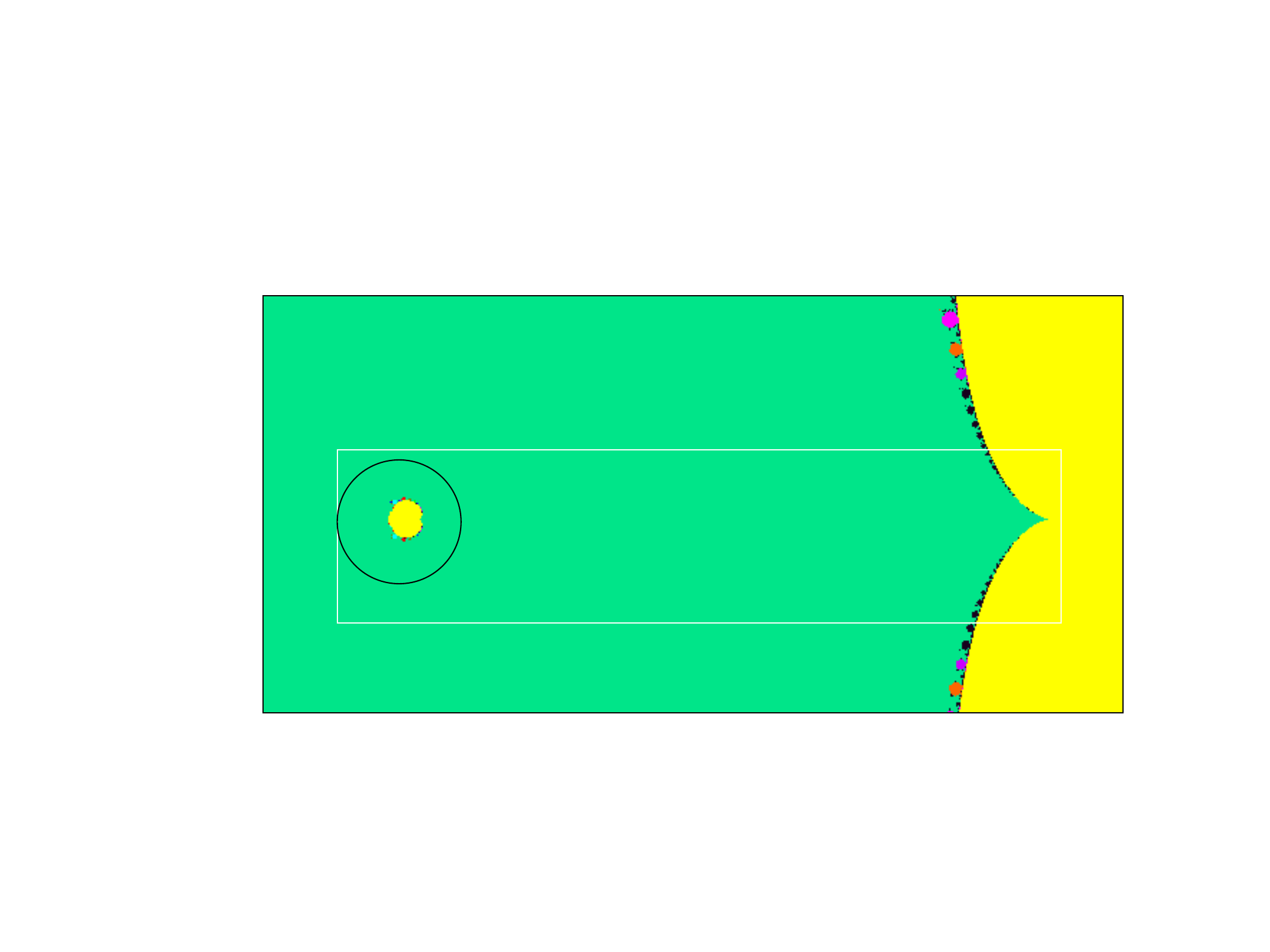}
  \caption{The $\la$ plane with the regions $\calm_{\la}$,  $\calm_{\mu}$ and the circle of inversion.}\label{lambda and rho}
  \label{circleofinversion}
\end{figure}

In theorem~\ref{calsla} we saw that $\cals_{\la}$ is homeomorphic to an annulus.  One of the complementary components is $\calm_{\la}$.   The other complementary component is bounded by the curve $\cals_*$.  By proposition~\ref{inversion},  $I$ maps $\cals_{\la} \cup \calm_{\la}$ to $\cals_{\mu} \cup \calm_{\mu}$;  since $I(\calm_{\la})=\calm_{\mu}$, $I(\cals_{\la})=\cals_{\mu}$ so that $\cals_{\mu}$ is also an annulus.    Because $I$ maps $\cals_*$ to itself, these annuli share a common boundary component.  

Note that although both the invariant circle of inversion $\calc_0$ and $\cals_*$ are invariant under inversion, unless $\rho$ is real,  they are not necessarily the same. 

Therefore $\cals \cup \{0\}=\cals_{\la} \cup \cals_{\mu} \cup \cals_* \cup \{ 0\}$ is topologically an annulus.
  Removing the parameter singularity $\la=0$  completes the proof.
\end{proof}

Immediate corollaries  of this theorem are:
\begin{cor} The sets $\calm_{\la}$ and $\calm_{\mu}$ are connected. \end{cor}
\begin{cor} The full shift locus in $\calf_2$ has the product structure $\DD^* \times \CC \setminus \{0,1 \}$. \end{cor}

\begin{figure}
 \centering
  \includegraphics[width=4in]{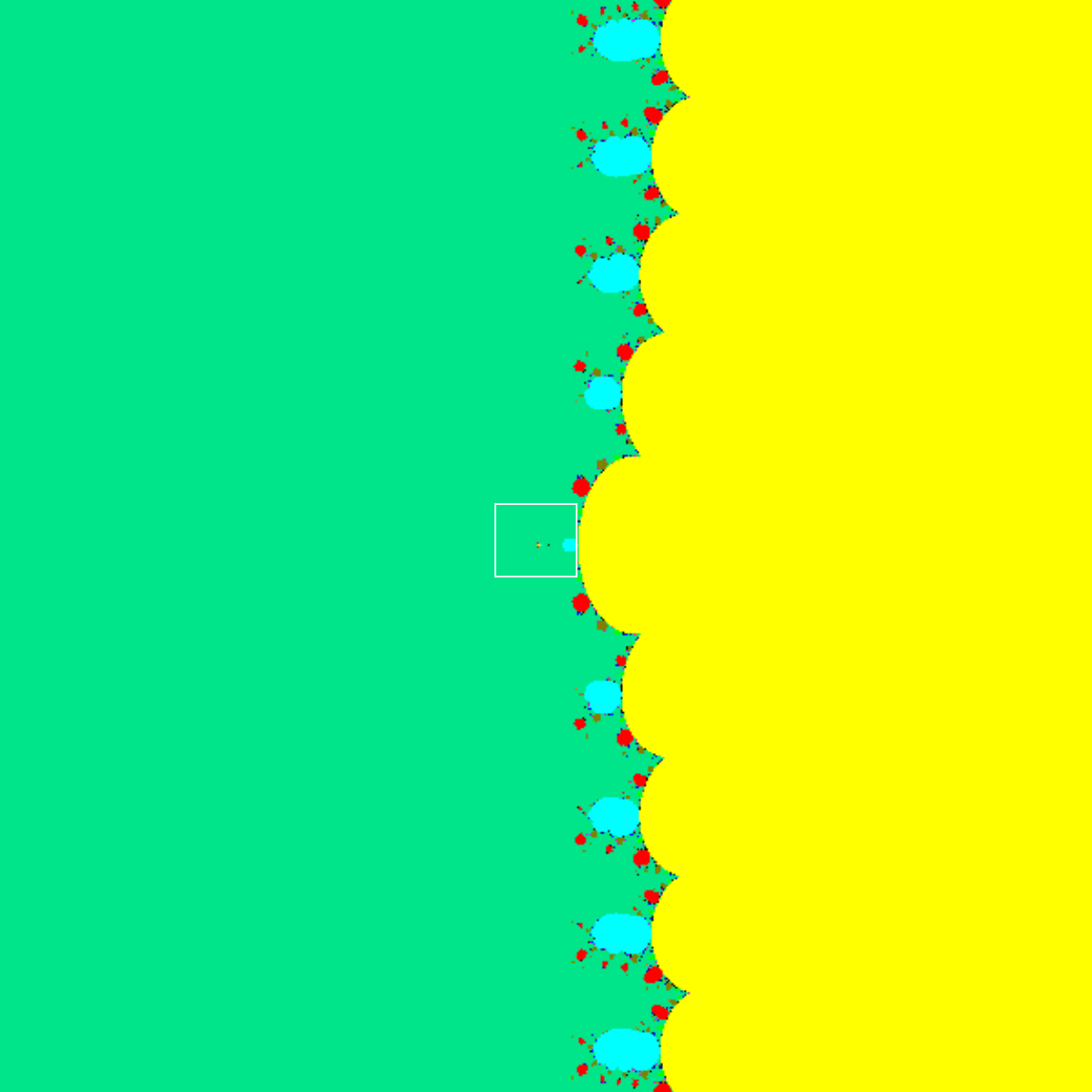}
  \caption{ The $\la$ plane when $\rho=-2/3$. Note the position of the period $2$ components. }\label{rho-2/3}
\end{figure}

 Figure(\ref{rho-2/3}) shows the $\la$ plane when $\rho=-2/3$. This is another slice in the fibration and shows how the fibers change as the argument of $\rho$ changes.   The picture is similar to figure~(\ref{lambdaplane}) except that we see that the $\calm_{\la}$ is translated vertically and there is a period $2$ component budding off $\Omega_1$ on the real axis instead of a cusp.

 \subsection{Single valued inverse branches}\label{Inverse branches}

We   now prove the lemma we assumed for the proof of the Combinatorial Structure Theorem in section~\ref{combinatorics}
\begin{lemma}\label{invwelldef}  There is a simply connected domain $\Sigma \in \CC \setminus \{0, \rho/2\}$ in which, after a choice of basepoint and branch of the logarithm,  the pole functions $p_k(\la)$ and the inverse branches $g_{\la,k}$ can be defined as single valued functions of $\la$.
\end{lemma}

\begin{proof} In the proof of theorem~\ref{thm:shift locus} we showed that $\cals_{\mu}$ and $\cals_{\la}$ are homeomorphic to annuli with a common boundary component that contains the singularity at the origin.   
It follows that in a neighborhood of the origin both asymptotic values are attracted to zero.

Now consider the period one shell component $\Omega_1'$ of $\calm_{\mu}$.  It has a virtual center at $\la=\rho/2$.    Since it is a virtual center, it is on the boundary of both $\calm_{\mu}$ and $\cals$ and so a neighborhood $V$ of $\rho/2$ contains points in $\cals_{\mu}$.

Applying the inversion,  $I(V)$ is a neighborhood of infinity intersecting the period one component $\Omega_1 \in \cals_{\la}$ and an open set in $\cals_{\la}$.  Thus  infinity is on the boundary   $\cals_{\la}$. Since a neighborhood of any point in $\cals_{*}$ only contains points in $\cals_{\la}$ and $\cals_{\mu}$, infinity and $0$ are on different boundary components.  Hence because $\cals_{\la}$ is an annulus,   we can find a curve   $\gamma \subset \cals_{\la}$ joining $0$ and infinity.    Let   $W$ be the component of $\CC \setminus \cals_{*}$ containing 
$\calm_{\mu}$ and set   $\Sigma = \CC \setminus (W \cup \gamma )$.  This is a simply connected domain.  It contains all the virtual cycle parameters belonging to $\calm_{\la}$ and none of the virtual cycle parameters belonging to $\calm_{\mu}$. Therefore, choosing a basepoint $\la_0 \in \Sigma$, and a branch for $\Log$, we can define  $p_k(\la_0)$ as in equation~(\ref{poles}) by
\[
 p_k(\lambda_0)=\frac{1}{2}\Log (\frac{\rho-2\lambda_0}{\rho})+ik\pi, 
\]
and extend analytically to all of $\Sigma$ as single valued functions of $\la$.  Then, as we did in section~\ref{combinatorics}, we can define the inverse branches of $f_{\la}$ as single valued functions of $\la$. 
\end{proof}

 \section{Concluding Remarks}
 \label{conclusion}
 There are many more questions one can address about the space of functions we have been studying.  Below we list some of them and leave an investigation of them to future work.

\begin{itemize}
\item
An important tool in studying the Mandelbrot set is the use of the level curves where the escape rate of the critical value is constant and their gradient ``external rays''.    Can we define the analogue for the set $\cals_{\la}$ and $\cals_{\mu}$ using the level curves of $\phi_0$ defined on $K_0$?  There will be infinitely many curves for each level so the structure will be much more complicated.   This would lead to more questions such as
\begin{enumerate}[i-]
\item  In \cite{CJK21}, we used the level curves and their gradients to prove that the virtual centers are accessible points from inside both the shell components and the shift locus.  Can we also  use it to  characterize  other  types of boundary points of $\cals$ such as  cusps,   root points for bud components or Misiurewicz points where an asymptotic value lands on a repelling cycle. 
\item Can we describe primitive and satellite components in terms of rays in a manner analogous to the discussion for rational maps.
\item In \cite{CJK19} we showed there is a renormalization operator defined for the family $it \tan z$ where $t$ is real.  Are there renormalization operators that can be defined in $\calf_2$? 
\end{enumerate}
\item
 We know that at the virtual centers and Misiurewicz points the only Fatou component is the attracting basin of the origin.   Is the Julia set a Cantor bouquet in the sense of Devaney?  Does it have positive measure? area?
\item  In \cite{GK} the mapping class group of the Teichm\"uller space $Rat_2$ is analyzed.  The analogous space here is $\calf_2$ from which points with orbit relations have been removed.   Describe the mapping class group of this space.
\item  How do the results here extend to  parameter spaces of families of meromorphic functions with more than two asymptotic values, or those with both critical values and asymptotic values.
\end{itemize}

\vspace*{20pt}
\noindent Tao Chen, Department of Mathematics, Engineering and Computer Science,
Laguardia Community College, CUNY,
31-10 Thomson Ave. Long Island City, NY 11101.
Email: tchen@lagcc.cuny.edu

\vspace*{5pt}
\noindent Yunping Jiang, Department of Mathematics, Queens College of CUNY,
Flushing, NY 11367 and Department of Mathematics, CUNY Graduate
School, New York, NY 10016
Email: yunping.jiang@qc.cuny.edu

\vspace*{5pt}
\noindent Linda Keen, Department of Mathematics,  CUNY Graduate
School, New York, NY 10016,
Email: LKeen@gc.cuny.edu; linda.keenbrezin@gmail.com

\end{document}